\documentclass[a4paper, 11pt] {article}
\usepackage{amsfonts, amsmath, amssymb, amsthm,url}
\usepackage{graphicx}
\usepackage[english] {babel}
\usepackage{enumitem}

\usepackage[ps, all,cmtip, dvips]{xy}

\usepackage{geometry}
\newtheorem{thm}{Theorem}[section]
\newtheorem{cor}[thm]{Corollary}
\newtheorem{lem}[thm]{Lemma}
\newtheorem{defi}[thm]{Definition}
\newtheorem{prop}[thm]{Proposition}

\theoremstyle{definition}
\newtheorem{rem}[thm]{Remark}

\DeclareMathOperator{\topo}{top}
\DeclareMathOperator{\Tor}{Tor}

\newcommand{\R}{\mathbb{R}}

\renewcommand{\Im}{\textrm{Im}}
\DeclareMathOperator{\Ker}{Ker}
\DeclareMathOperator{\Ima}{Im}
\newcommand{\N}{\mathbb{N}}

\newcommand{\MT}{\mathcal{MT}}
\newcommand{\Z}{\mathbb{Z}}

\DeclareMathOperator{\Th}{Th}

\DeclareMathOperator{\Ext}{Ext}
\DeclareMathOperator{\Hom}{Hom}

\DeclareMathOperator{\Sq}{Sq}
\newcommand{\A}{\mathcal{A}}
\newcommand{\V}{\mathcal{V}}

\newcommand*\InsertTheoremBreak{
\begingroup 
\setlength\itemsep{0pt}
\setlength\parsep{0pt}
\item[\vbox{\null}]
\endgroup
}

\title{A generalization of the stable EHP spectral sequence}
\author{Marcel B\"{o}kstedt and Anne Marie Svane}
\date{}
\begin{document} 
\maketitle
\begin{abstract}
For any vector bundle, we define an inverse system of spectra. In the case of a trivial bundle over a point, the homotopy groups of the filtration quotients give rise to the stable EHP spectral sequence, as was shown by Mahowald in \cite{mahowald}. The limit was determined by Lin in \cite{lin}. We also obtain a spectral sequence and show strong convergence in certain special cases. For compact base spaces, we obtain a generalization of Lin's theorem. In the case of the universal bundle over $BO$, we can also determine the limit by means of an Adams spectral sequence. This turns out to be quite different from the compact case. We also obtain partial results for the universal bundle over $BSO$. 
\end{abstract}
\section{Introduction}
The well-known EHP spectral sequence is usually constructed from the long exact sequences of homotopy groups associated to the filtration
\begin{equation*}
\Omega S^1 \to \Omega^2 S^2 \to \dotsm.
\end{equation*}
This spectral sequence exhibits a periodic behaviour in a certain range. This naturally leads to the construction of the stable EHP spectral sequence given by ex\-ten\-ding periodically. 

In \cite{mahowald}, Mahowald constructed the stable EHP spectral sequence rather by considering the Atiyah--Hirzebruch spectral sequence for the suspension spectra $\Sigma^\infty P_{d,r}$ of the stunted projective space $P_{d,r} = \R P^{d-1}/\R P^{d-r-1}$. 
By the James homeomorphisms $\Sigma^{a_r} P_{d,r} \to P_{d+a_r,r}$, see e.g.\ \cite{james2}, the suspension spectra $\Sigma^\infty P_{d,r}$ are well-defined even for $d<0$ and $r\geq d$ by interpreting $\Sigma^{ka_{r}}P_{d,r}$ as $P_{d+ka_{r},r}$ for $k$ sufficiently large. 
 The number $a_{r+1}$ is the dimension of an irreducible representation of the Clifford algebra $Cl_r$, see e.g.\ \cite{lawson}. It is given by the table
\begin{equation}\label{RH}
\begin{tabular}{|c|c|c|c|c|c|c|c|c|}
\hline 
{$r$}&{1}&{2}&{3}&{4}&{5}&{6}&{7}&{8}\\
\hline
{$a_r$}&{1}&{2}&{4}&{4}&{8}&{8}&{8}&{8}\\
\hline
\end{tabular}
\end{equation}
for $r \leq 8$, and in general $a_{r+8} = 16 a_r$. 

The projections $P_{d+ka_{r+1},r+1} \to P_{d+ka_{r+1},r}$ define an inverse system of spectra 
\begin{equation}\label{Plim}
\dotsm \to \Sigma^{\infty } \Sigma P_{d,r+1} \to \Sigma^\infty \Sigma P_{d,r} \to \dotsm \to \Sigma^\infty \Sigma P_{d,1}
\end{equation}
Each map induces a map of spectral sequences, and Mahowald considered the spectral sequence obtained in the limit.

Mahowald's spectral sequence converges to $ \varprojlim_r \pi_*(\Sigma^{\infty +1} P_{d,r})$.
Mahowald conjectured that the homotopy limit $ \varprojlim_r \Sigma^{\infty +1} P_{d,r}$ was the 2-completed sphere spectrum $({S}^{0})_2^\wedge$ in dimensions $*<d-1$. This was later proved by Lin in \cite{lin}. Letting $d$ tend to infinity, one obtains the stable EHP spectral sequence, and it follows that this converges to $\pi_*(({S}^{0})_2^\wedge)$.

The initial motivation for this paper was certain spectra $MT(d,r)$ defined in \cite{adb}, related to the cobordism category of \cite{GMTW} and to the existence of trivial subbundles of vector bundles in \cite{adb}. Certain low-dimensional homotopy groups were computed in \cite{adb}. The computations can be summarized in a spectral sequence. We do this in Section \ref{thess}. 

The spectra $MT(d,r)$ satisfy a version of James periodicity, see \cite{adb}, Section 3.3. Hence, in a certain range the spectral sequence behaves periodically. Thus it seems natural to let $r$ tend to infinity by periodicity.

In Section \ref{insy} we are going to construct a version of \eqref{Plim} that naturally generalizes to any vector bundle $E\to X$ classified by a map $f:X \to BO(d)$. We obtain an inverse system
\begin{equation}\label{fMlim}
\dotsm \to f^*\MT(d,r+1)\to f^*\MT(d,r)\to \dotsm \to f^*\MT(d,1)
\end{equation}
of spectra.
We consider in particular the case where $X$ is compact and the one where it is $BO$, $BSO$, or $BSpin$. The case where $X$ is a point corresponds to \eqref{Plim}. 

There is a spectral sequence for each of the spectra $f^*\MT(d,r)$.
As in Mahowald's case, we obtain a spectral sequence in the limit.
In Section~\ref{sss} we show that this spectral sequence converges strongly to $\varprojlim_r\pi_*(f^*\MT(d,r))^\wedge_2$.  
When $X$ is a point, we recover the stable EHP spectral sequence as constructed by Mahowald. 
In the $BO$, $BSO$, and $BSpin$ cases, we obtain exactly the stabilization of the spectral sequence for $MT(d,r)$ suggested in Section~\ref{thess}.

The remainder of this paper is concerned with the description of the homotopy limit $\varprojlim_r f^*\MT(d,r)$.
The case where $X$ is compact is the topic of Section~\ref{pbs}. We shall see that the limit is $\Sigma^{\infty-n}\Th(N)_2^\wedge$, i.e.\ the 2-completion of the suspension  spectrum of the Thom space of an $n$-dimensional complement $N$ of $E$. We give two proofs of this. The first proof is a topological argument using induction on the number of cells in $X$, starting with Lin's theorem. The other one is more algebraic. 

Recall that $H^*(\R P^{d-1};\Z/2) \cong \Z/2[t]/(t^{d})$, i.e.\ the truncated polynomial algebra on the generator $t$.
On $\Z/2$ cohomology, the sequence \eqref{Plim} induces a direct sequence
\begin{equation*}
\Z/2\{t^{d-r+1},\dots ,t^{d}\} \to \Z/2\{t^{d-r},\dots ,t^{d}\} \to \dotsm.
\end{equation*}
Here $\Z/2\{x^1,\dots ,x^k\}$ denotes the graded $\Z/2$-vector space with basis elements $x^i$ in dimension $i$.
The direct limit of these cohomology groups is $\Z/2\{t^l,l\leq d\}$, i.e.\ the Laurent polynomials in the variable $t$ of degree at most $d$. 
There are maps of spectra $ S^{0} \to \Sigma^{\infty} \Sigma P_{d,r}$ commuting with the maps in \eqref{Plim} and thus  inducing a map $\hat{\varphi}_0: S^{0} \to \varprojlim_r \Sigma^\infty \Sigma P_{d,r}$. 
Lin's proof is based on the observation that the induced map on cohomology 
\begin{equation}\label{isoext}
\hat{\varphi}^*_0:\Z/2\{t^l,l\leq d\} \to  \Z/2
\end{equation}
induces an isomorphism on $\Ext_\A^{s,t}$ in degrees $t-s\leq d$. 
A generalization of this, known as the Singer construction, yields the algebraic proof for general vector bundles.

In the last Section \ref{nonres}, we consider the cases $X=BO$ and $X=BSO$. From the above, one would expect the limit of \eqref{fMlim} to be a 2-completion of the Thom cobordism spectra $MO$ and $MSO$, respectively. However, we get a completely different result in this case. The limit is not even connected anymore. In fact, the homotopy groups are uncountable in all dimensions. We determine these completely in the $BO$ case and partially in the $BSO$ case. 

The proof goes by studying a version of the Adams spectral sequence for inverse limits. This is the direct limit as $r$ tends to infinity of the Adams spectral sequences used to compute $\pi_*(MT(d,r))$ in \cite{adb} for low values of $r$. We determine this direct limit and in the $BO$ case, we show that there are no non-trivial differentials in the limit.

\section{The inverse systems}\label{insy}
In this first section we are going to construct the inverse system \eqref{fMlim}. We first do so in the $X=pt$ case in Section \ref{dcls} and then generalize this to maps $f: X \to BO$ in Section \ref{dils}. It also serves as an introduction of notation.

\subsection{An alternative construction of Mahowald's spectrum}\label{dcls}
Let $V_{d,r}$ be the Stiefel manifold consisting of ordered $r$-tuples of orthonormal vectors in $\R^d$. Similarly, let $W_{d,r}$ denote the cone on $V_{d,r}$. A point in $W_{d,r}$ may be thought of as an $r$-tuple $(v_1,\dots,v_r)$
of orthogonal vectors of length $0\leq |v_1|=\dots =|v_r|\leq 1$.

James has constructed a $2(d-r)$-equivalence $P_{d,r} \to V_{d,r}$, see \cite{james}. We will use this to construct a version of $\Sigma^{\infty +1}P_{d,r}$ that is independent of the coordinates in $\R^{d}$ and hence easier to generalize to vector bundles.

Let $V$ be a representation of the Clifford algebra $Cl_{r-1}$ of dimension ${ka_r}$ for some $k$. This restricts to a multiplication $\R \oplus \R^{r-1}\times V\to V$. Let $e_1,\dots ,e_{r}$ be an orthonormal basis for $\R\oplus \R^r$. For a suitable inner product on $V$, the multiplication will satisfy that for any $x\in V$, $e_1x,\dots ,e_rx$ are orthogonal of the same length and $e_1x=x$, see \cite{lawson}. With this inner product, we identify $V$ with $\R^{ka_r}$.

Using such an orthogonal multiplication $\R^r \times \R^{ka_r} \to \R^{ka_r}$, define a map of pairs 
\begin{equation} \label{Vmap}
g_0 : (W_{d,r}, V_{d,r}) \times  (D^{ka_r},S^{ka_r-1}) \to (W_{d+ka_r,r},V_{d+ka_r,r})
\end{equation}
 by
\begin{equation*}
(v_1,\dots ,v_{r},x) \mapsto (   \sqrt{1-|x|^2}v_1 + e_1x,\dots,  \sqrt{1-|x|^2}v_{r} + e_{r}x).
\end{equation*} 
This defines a periodicity map 
\begin{equation*}
g_{0}: \Sigma^{ka_r} \Sigma V_{d,r} \to \Sigma V_{d+ka_r,r}.
\end{equation*}
It was shown in \cite{adb} that $g_0$ is a $(2(d-r)+ka_r+1)$-equivalence. 
 
The periodicity maps form a direct system of spectra
\begin{equation*}
\Sigma^\infty \Sigma V_{d,r} \to  \Sigma^{\infty -ka_r} \Sigma V_{d+ka_r,r} \to \Sigma^{\infty-2ka_r } \Sigma V_{d+2ka_r,r} \to \dotsm .
\end{equation*}
Taking the homotopy direct limit of these spectra yields a spectrum which we denote by $\mathcal{V}_{d,r}$. This satisfies
\begin{equation*}
\pi_q(\mathcal{V}_{d,r}) \cong \varinjlim_l{\pi_{q + lka_r}^s(\Sigma V_{d+lka_r,r})}.
\end{equation*} 
Note that this is defined even when $d$ is negative or $ r>d $ just by starting the direct sequence at $ \Sigma^{\infty -lka_r } \Sigma V_{d+lka_r,r}$ for some large $l$. 

\begin{thm}\label{PVeq}
The composite map  
\begin{equation*}
\Sigma^{\infty} \Sigma P_{d,r}\cong \Sigma^{\infty-ka_r} \Sigma P_{d+ka_r,r} \to \Sigma^{\infty-ka_r} \Sigma V_{d+ka_r,r} \to \mathcal{V}_{d,r}
\end{equation*}
defines a homotopy equivalence if $2r < d+ka_r$.
\end{thm}

\begin{proof}
Assume $2r<d$.
The James map $P_{d,r} \to V_{d,r}$ is $2(d-r)$-connected. Thus it induces an isomorphism on homology in dimensions less than $2(d-r)>d$. The map $\Sigma^{ka_r} \Sigma V_{d,r} \to \Sigma V_{d+ka_r, r}$ is $(2(d-r+1) + ka_r)$-connected. Thus the composite map 
\begin{equation}\label{comp}
\Sigma^{ka_r} \Sigma P_{d,r} \to \Sigma^{ka_r} \Sigma V_{d,r} \to \Sigma V_{d+ka_r,r}
\end{equation}
induces an isomorphism on homology in dimensions up to $d + ka_r$. Both $\Sigma^{ka_r} \Sigma P_{d,r}$ and $\Sigma V_{d+ka_r,r}$ have homology zero in dimensions between $d + ka_r$ and $2(d-r + ka_r)$, since they have no cells in these dimensions, see e.g.\ \cite{hatcher}, Chapter~3.D. Thus  \eqref{comp} is $(2(d-r) + ka_r)$-connected. Letting $l$ tend to infinity, we get an isomorphism
\begin{equation*}
\pi_*^s(\Sigma P_{d,r}) \to \varinjlim_k \pi_{* + ka_r}^s(\Sigma V_{d+ka_r,r}) = \pi_*(\V_{d,r}) .
\end{equation*}
Thus the map $\Sigma^{\infty +1}P_{d,r} \to \V_{d,r}$ must be a homotopy equivalence.
%
\end{proof}

The  following diagram of Stiefel manifolds
\begin{equation}\label{ustabil}
\vcenter{\xymatrix{{V_{d,r+1}}\ar[d]\ar[r]&{V_{d,{r}}}\ar[d]\\
{V_{d+1,r+1}}\ar[r]&{V_{d+1,r}}
}}
\end{equation}
commutes. The vertical maps come from the splitting $\R^{d+1}\cong \R^d \oplus \R$ and the horizontal maps come from forgetting the $(r+1)$th vector. We want to see that this induces a well-defined commutative diagram on direct limits
\begin{equation}\label{system}
\vcenter{\xymatrix{{\V_{d,r+1}}\ar[d]\ar[r]&{\V_{d,{r}}}\ar[d]\\
{\V_{d+1,r+1}}\ar[r]&{\V_{d+1,r}.}
}}
\end{equation}
So far, we have made no assumptions on the actual choice of orthogonal multiplication in the construction of the periodicity maps. However, this becomes important if we want to make the diagram strictly commutative.
\begin{lem} \label{henvis}
Let $m_{r+1}:\R^{r+1}\times \R^{ka_{r+1}}\to \R^{ka_{r+1}}$ be an orthogonal multiplication. 
Define $\V_{d,r+1}$ and $\V_{d,r}$ using $m_{r+1}$ and its resriction to $(\R^{r}\oplus 0) \times \R^{ka_{r+1}}$, respectively. 
Then the diagram \eqref{system} is well-defined and commutative. 
\end{lem}

\begin{proof}
It is easy to see from the formulas that the diagrams
\begin{equation*}
\xymatrix{{\Sigma^{a_r}\Sigma V_{d,r}}\ar[r] \ar[d]^{}&{\Sigma V_{d+a_r,r}}\ar[d]^{}&{\Sigma^{a_{r+1}}\Sigma V_{d,r+1}} \ar[r]\ar[d]& {\Sigma V_{d+a_{r+1},r+1}}\ar[d]\\
{\Sigma^{a_r}\Sigma V_{d+1,r}}\ar[r]&{\Sigma V_{d+1+a_r,r}}& {\Sigma^{a_{r+1}}\Sigma V_{d,r}}\ar[r]&{\Sigma V_{d+a_{r+1},r}}
}
\end{equation*}
commute. Here the vertical maps in the first diagram come from the inclusion $\R^d \oplus \R^{a_r} \subseteq \R^d \oplus \R \oplus \R^{a_r}$ and in the second diagram they come from forgetting the last vector. Hence the maps  $\V_{d,r} \to \V_{d+1,r}$ and $\V_{d,r+1} \to \V_{d,r}$ are well-defined.

Commutativity of \eqref{system} now follows because the diagram
\begin{equation*}
\xymatrix{{V_{d + ka_{r+1},r+1}}\ar[d]^{}\ar[r]&{V_{d+ka_{r+1},{r}}}\ar[d]^{}\\
{V_{d+1 + ka_{r+1},r+1}}\ar[r]&{V_{d+1 +ka_{r+1},r}}
}
\end{equation*}
commutes for all $k$.
\end{proof}

\begin{lem}
For each $d$, there is an inverse system of spectra
\begin{equation}\label{insyV}
\dotsm \to \V_{d,r+1} \to \V_{d,r}.
\end{equation}
The inclusions $\V_{d,r} \to \V_{d+1,r}$ define a map of inverse systems.
\end{lem}

\begin{proof}
We first show that we can choose multiplications  
\begin{equation*}
1 \oplus m_r : (\R \oplus \R^{r}) \times \R^{16^k}  \to \R^{16^k}
\end{equation*}
 for each $r$, such that the first $\R$-summand acts as the identity and $m_{r}(a,x)$ is orthogonal to $x$ and has length $|a||x|$ for all $a\in \R^r$ and $x \in \R^{16^k}$, and such that the restriction of $1 \oplus m_r$ to $\R \oplus \R^{r-1}$ is an orthogonal direct sum of the chosen $1\oplus m_{r-1}$. 

Suppose given such multiplications $m_{8s}: \R^{8s} \times \R^{16^s}\to \R^{16^s}$ for all $s\leq k$. Then there is a multiplication 
\begin{equation*}
m_{8(k+1)}:( \R^{8k} \oplus \R^8) \times (\R^{16^k} \otimes \R^{16})\to \R^{16^k} \otimes \R^{16}
\end{equation*}
 given by $m_{8k} \otimes 1 + 1\otimes m_8$ satisfying orthogonality when $\R^{16^k} \otimes \R^{16}$ is given the inner product
\begin{equation*}
\langle x\otimes y, x'\otimes y' \rangle = \langle x,x'\rangle \langle y,y'\rangle.
\end{equation*}
For a suitable isometry $\R^{16^k} \otimes \R^{16}\cong \bigoplus_{i=1}^{16} \R^{16^k}$, the restriction to $ \R^{8k}\oplus 0$ is the orthogonal sum of 16 copies of $m_{8k}$.

This allows us to  choose multiplications $m_{8(k+1)}$ inductively such that each restricts
to an orthogonal sum of 16 copies of $m_{8k}$. For all $8k <r <8(k+1)$, we choose the multiplication $m_r:\R^{r}  \times \R^{16^{k+1}}  \to \R^{16^{k+1}}$ to be the restriction of $m_{8(k+1)}$. 

With these choices, the diagram \eqref{system} commutes for all $r\neq 8k$, and for $r=8k$, it is enough to check that the periodicity map defined by an orthogonal sum of two multiplications is the same as the composition of the two periodicity maps corresponding to each of the two multiplications. This is straightforward.
\end{proof}
\begin{defi} 
\begin{equation*}
\begin{split}
\widehat{V}_d &= \varprojlim_r \V_{d,r}\\
\widehat{V} &= \varinjlim_d \widehat{V}_d.
\end{split}
\end{equation*}
\end{defi}

For $d\geq 0$, the orthogonal multiplication defines a map $\varphi_0 :S^{ka_r-1}\to V_{d+ka_r,r}$ by 
\begin{equation*}
\varphi_0(x)= (e_1x,\dots,e_{r}x) \in 0\oplus \R^{ka_r}\subseteq \R^{d+ka_r}.
\end{equation*}
This yields a map of spectra $\varphi_0 : S^{0} \to \V_{d,r}$ for each $d$ and $r$ such that
\begin{equation*}
\xymatrix{{\V_{d,r+1}} \ar[r]&{\V_{d+1,r+1}}\ar[d]\\
{S^{0}}\ar[ur]^{\varphi_0} \ar[r]^{\varphi_0} \ar[u]^{\varphi_0}&{\V_{d+1,r}}
}
\end{equation*}
commutes. Thus there are maps 
\begin{equation*}
\begin{split}
\hat{\varphi}_0 &: S^0 \to \widehat{V}_d\\
\hat{\varphi}_0 &: S^0 \to \widehat{V}.
\end{split}
\end{equation*}

The inverse system \eqref{insyV} is essentially \eqref{Plim}. We have not been able to construct a homotopy commutative map between the two diagrams, but \eqref{insyV} does satisfy the following version of Lin's theorem:
\begin{thm}\label{linV}
\InsertTheoremBreak
\begin{itemize}
\item[(i)] $\varprojlim_r \pi_q(\V_{d,r}) = 0$ when $q<0$ and $q<d$.
\item[(ii)] For  $0<d-1$, $\varprojlim_r \pi_0(\V_{d,r}) \cong \Z_2^\wedge$ as topological groups. The inclusion
\begin{equation*}
\hat{\varphi}_{0*}:\Z = \pi_0(S^{0}) \to \varprojlim_r \pi_0(\V_{d,r})
\end{equation*}
is non-zero mod $2$.
\item[(iii)] When $q>0$ and  $d>0$, the map induced by $\hat{\varphi}_0$
\begin{equation*}
\hat{\varphi}_{0*}: \pi_q(S^0)_2^\wedge \to \varprojlim_r \pi_q(\V_{d,r})
\end{equation*}
is an isomorphism when $q<d-1$ and surjective when $q=d-1$.
\end{itemize}
\end{thm}
Here $G_2^\wedge$ denotes the 2-completion $\varprojlim_r G/2^r G$ of the group $G$.
\begin{proof}
In \cite{lin}, Lin proves the corresponding theorem for the homotopy inverse limit of \eqref{Plim}.  The proof of Lin's theorem only relies on the structure of $\varinjlim_r H^*(\Sigma P_{d,r})$ as a module over the Steenrod algebra and the fact that each spectrum in the inverse system has an S-dual. 
Since $\varinjlim_r H^*(\V_{d,r})\cong \varinjlim_r H^*(\Sigma P_{d,r})$ as Steenrod modules, the proof carries over to our situation.

To see that $\varphi_{0*}$ is as claimed, we need $\varphi_0^{*}: \varinjlim_r H^0(\V_{d,r};\Z/2) \to H^0(S^0;\Z/2)$ to be non-zero. But $\varphi_0$ factors as 
\begin{equation*}
 S^{ka_r-1} \to V_{ka_r,r} \to  V_{d+ka_r,r}.
\end{equation*} 
The composition $S^{ka_r-1} \to V_{ka_r,r} \to S^{ka_r -1}$ is the identity. Thus the first map is an isomorphism on $H^{ka_r-1}$. So is the second map if $r>d$. 

Alternatively, one could prove the theorem by referring to \cite{rognes}, Proposition 2.2. This is the approach in Section \ref{singerc}.
\end{proof}

From now on, a finite 2-primary group will mean a finite abelian group where all elements have order some power of 2. A 2-profinite group is an inverse limit of finite 2-primary groups. 

\begin{rem}\label{improveV}
For $d-r$ and $p$ odd, $H^q(\Sigma V_{d,r};\Z/p)=0$ for $q<d$ or $q=d$ odd. Thus $\pi_q^s(V_{d,r})$ must be a finite 2-primary group. Since $\pi_q(\V_{d,r}) \cong \pi_{q+ka_r}^s(V_{d+ka_r,r})$ for some even $ka_r$, $\pi_q(\V_{d,r})$ is also a finite 2-primary group for $d-r$ odd. 


Therefore, by the Mittag--Leffler condition, $\varprojlim_r\nolimits^1 \pi_q(\V_{d,r})=0$ for those $q$, and hence $\pi_q(\widehat{V}_{d}) \cong \varprojlim_r \pi_q(\V_{d,r})$ for $q<d-1$. 
\end{rem}

\subsection{Generalization to Thom spectra}\label{dils}
Let $G(d,n)$ denote the Grassmannian consisting of $d$-dimensional subspaces of $\R^{n+d}$. Let $U_{d,n}\to G(d,n)$ be the universal bundle with $n$-dimensional orthogonal complement $U_{d,n}^\perp$. Then $MTO(d)$ is the spectrum with $n$th space
\begin{equation*}
MTO(d)_n=\Th(U_{d,n}^\perp)
\end{equation*} 
where $\Th(\cdot)$ denotes the Thom space.
The splitting $\R^{1+n +d} = \R\oplus \R^{n+d}$ defines an inclusion $G(d,n)\to G(d,n+1)$ and the restriction of $U_{d,n+1}^\perp$ to $G(d,n)$ is $\R \oplus U_{d,n}^\perp$. This defines spectrum maps
\begin{equation*}
\Sigma\Th(U_{d,n}^\perp) = \Th(\R \oplus U_{d,n}^\perp) \to \Th(U_{d,n+1}^\perp).
\end{equation*}

More generally, let $f: X \to BO$ be any map. Then $X$ is filtered by the subspaces $X_{d,n} = f^{-1}(G(d,n))$. Again this defines a spectrum $f^*MT(d)$ with 
\begin{equation*}
f^*MT(d)_n = \Th(f^*U_{d,n}^\perp \to X_{d,n}).
\end{equation*}

Let $E= f^*U_d$. There is a fiber bundle
\begin{equation*}
V_{d,r} \to V_r(E) \xrightarrow{p_{V_r}} X
\end{equation*}
where the fiber over $x\in X$ is the set of ordered $r$-tuples of orthonormal vectors in the fiber $E_x$.  
Similarly there is a fibration 
\begin{equation*}
W_{d,r}\to W_r(E) \xrightarrow{p_{W_r}} X
\end{equation*}
with contractible fiber $W_{d,r}$. A point in $p_{W_r}^{-1}(x)$ is an $r$-tuple of orthogonal vectors in $E_x$ of common length at most one.

The above construction yields an inclusion
\begin{equation}\label{defdr}
(f\circ p_{V_r})^*MT(d) \to (f\circ p_{W_r})^*MT(d). 
\end{equation}

\begin{defi}
Let $f^*MT(d,r)$ denote the cofiber of  \eqref{defdr}.
\end{defi}

When $X$ is $BO$, $BSO$, or $BSpin$, we shall sometimes denote these spectra by $MTO(d,r)$, $MTSO(d,r)$, and $MTSpin(d,r)$, respectively. When $Y$ is a subspace of $X$, we sometimes use the notation $f^*MT(d,r)_{\mid Y}$.
Apart from these examples, we are mainly interested in the case where $X$ is compact. 

The definition of the periodicity map extends to a map  
\begin{equation*}
\Sigma^{ka_r} f^*MT(d,r) \to f^*MT(d+ka_r,r).
\end{equation*}
This is defined by the map of pairs
\begin{align*}
g:  ( W_{r}(f^*U_{d,n}), V_{r}(f^*U_{d,n}))\times (D^{ka_r},S^{ka_r-1}) \to (W_{r}(f^*U_{d+ka_r,n}),V_{r}(f^*U_{d+ka_r,n})),
\end{align*}
given on $x\in X_{d,n}$, $(v_1,\dots,v_r)\in f(x)$, and $t\in D^{ka_r}$ by 
\begin{align*}
g(x,v_1,\dots,v_r,t) = (x, \sqrt{1-|t|^2}v_1 + e_1t,\dots ,\sqrt{1-|t|^2}v_{r} + e_{r}t).
\end{align*}
That is, $g$ is just a fiberwise application of the map $g_0$.

As before, we define the direct limit 
\begin{equation*}
f^*\MT(d,r) = \varinjlim_l \Sigma^{-lka_r} f^*MT(d+lka_r,r) .
\end{equation*}
 
The same commutativity results leading to the diagram \eqref{system} immediately yields that 
\begin{equation*}
\xymatrix{{f^*\MT(d,r+1)}\ar[r]\ar[d]&{f^*\MT(d,r)}\ar[d]\\
{f^*\MT(d+1,r´+1)}\ar[r]&{f^*\MT(d+1,r)}
}
\end{equation*}
is well-defined and commutative. 
The vertical maps are defined using the inclusion $G(d+ka_r ,n) \to G(d+1 + ka_r,n)$ coming from $\R^n \oplus \R^{d} \oplus \R^{ka_r} \subseteq \R^n \oplus \R^{d} \oplus \R \oplus \R^{ka_r}$. 
Again we may define:
\begin{defi}
\begin{equation*}
\begin{split}
\widehat{f^*MT}(d) &= \varprojlim_r f^*\MT(d,r)\\
\widehat{f^*MT} &= \varinjlim_d \widehat{f^*MT}(d). 
\end{split}
\end{equation*}
\end{defi}

\begin{prop}
The construction is natural, i.e.\ a composition $X\xrightarrow{f} Y \xrightarrow{g} BO$ induces a map $f_* : (g\circ f)^*\MT(d,r) \to g^*\MT(d,r)$. Furthermore,
\begin{equation*}
\begin{split}
\V_{d,r}&=\MT(d,r)_{\mid pt} \\
\mathcal{MTO}(d,r)&=\varinjlim_{X} \mathcal{MTO}(d,r)_{\mid X}
\end{split}
\end{equation*}
where the direct limit is taken over compact $X\subseteq BO$.
\end{prop}
However, direct and inverse limits do not commute in general, so we cannot expect
\begin{equation*}
\varinjlim_{X} \widehat{MTO}(d)_{\mid X} \cong \varinjlim_{X}  \varprojlim_r \mathcal{MTO}(d,r)_{\mid X}  \cong  \varprojlim_r \varinjlim_{X} \mathcal{MTO}(d,r)_{\mid X} \cong \widehat{MTO}(d).
\end{equation*}
In fact, as we shall see in Section \ref{nonres}, this is not at all the case, since the right hand side is connected while the left hand side is not.

The definition of  $\hat{\varphi}_0: S^0 \to \widehat{V}_d$ naturally extends to a  map 
\begin{equation*}
\widehat{f^*\varphi} :f^*MT(d) \to \widehat{f^*MT}(d)
\end{equation*}
such  that the diagram 
\begin{equation}\label{sam}
\vcenter{\xymatrix{{S^0}\ar[r]^{\hat{\varphi}_0}\ar[d]&{\widehat{V}_d}\ar[d]\\
{f^*MT(d)}\ar[r]^{\widehat{f^*\varphi}}&{\widehat{f^*MT}(d)}
}}
\end{equation}
commutes. The vertical maps come from the inclusion $pt \to X$.

The map $\widehat{f^*\varphi}$ is defined as follows.
First we  define 
\begin{equation*}
\Sigma^{ka_r}f^*MT(d) \to f^*MT(d+ka_r,r). 
\end{equation*}
Recall that  $\Sigma^{ka_r}f^*MT(d)$ is the quotient of
\begin{equation*}
(B^{ka_r}, S^{ka_r - 1})\times (W_1(f^*U_{d,n}^\perp),V_1(f^*U_{d,n}^\perp) ).
\end{equation*}

If $p: W_r(f^*U_{d+ka_r,n}) \to X_{d+ka_r,n}$, we can think of $f^*MT(d+ka_r,r)$ as the quotient of 
\begin{equation*}
(p^*W_1(f^*U_{d+ka_r,n}^\perp),p^*V_1(f^*U_{d+ka_r,n}^\perp) \cup p^*W_1(f^*U_{d+ka_r,n}^\perp)_{\mid V_r(f^*U_{d+ka_r,n})}).
\end{equation*}
Let $(t,x,v) \in B^{ka_r}\times W_1(f^*U_{d,n}^\perp) $ where $t \in B^{ka_r}$, $x \in X_{d,n}$, and $v \in f(x)^\perp$. This should be mapped to $(e_1t,\dots,e_{r}t,x,v)$  
 where $x \in X_{d,n}\subseteq X_{d+ka_r,n}$, $(e_1t,\dots,e_{r}t)$ is a frame in $0\oplus \R^{ka_r} \subseteq f(x)\oplus \R^{ka_r}$, and  $v \in (f(x)\oplus \R^{ka_r})^\perp$. It is easy to see that this map commutes with all the relevant maps in the limit systems and thus defines the desired map $\widehat{f^*\varphi}$.

%

\begin{prop} \label{limtheta}
The map induced by $pt \to BSO(d)$,
\begin{equation*}
 \pi_q( \widehat{V}_{d}) \to \pi_q( \widehat{MTSO}(d)),
\end{equation*}
is zero for $q<d$ and $q \neq 0$.
\end{prop}

\begin{proof}
The map $\pi_q(S^0) \to \pi_q(MTSO(d))$ is zero in the relevant dimensions, being the inclusion of the framed cobordism group into the oriented cobordism group. So since $\pi_q(S^0) \to \pi_q(\widehat{V}_{d})$ is surjective in these dimensions, the claim follows from the diagram~\eqref{sam}. 
\end{proof}

\section{Application to spectral sequences}
In the first Section \ref{thess}, we consider a spectral sequence converging to $\pi_*(MT(d,r))$ and identify the first differential. Then in Section \ref{sss}, we stabilize this spectral sequence and let $r$ tend to infinity and we show that the resulting spectral sequence converges strongly. In the $X=pt$ case, we show that we obtain the stable EHP spectral sequence.

\subsection{A spectral sequence for $MT(d,r)$}\label{thess}
Let $B(d)$ denote either $BO(d)$ or $BSO(d)$ and let $MT$ denote the corresponding spectra. 
Consider the inclusion $MT(d-r) \to MT(d)$ coming from the splitting $\R^{n+d}\cong \R^{n+d-r}\oplus \R^r$. 
This inclusion is filtered as the composition
\begin{equation}\label{filt}
MT(d-r) \to MT(d-r+1) \to \dotsm \to MT(d-1) \to MT(d).
\end{equation}

There is a homotopy equivalence $B(d-r) \to V_r(U_d)$ given by mapping $V\subseteq \R^{\infty + d}$ to $V\oplus \R^r \subseteq \R^{\infty + d+r}$ with the orthogonal $r$ frame given by the standard basis in $0\oplus \R^r$.
This extends to a commutative diagram
\begin{equation*}
\xymatrix{{MT(d-r)}\ar[r]\ar[d]&{MT(d)}\ar[d]\\
{p_{V_r}^*MT(d)}\ar[r]&{p_{W_r}^*MT(d)}
}
\end{equation*}
where the vertical maps are homotopy equivalences. In particular:
\begin{lem}\label{cofdr}
The cofiber of the inclusion $MT(d-r)\to MT(d)$ is homotopy equivalent to $MT(d,r)$.  
\end{lem}
In particular, each map in the filtration \eqref{filt} fits into a cofibration sequence
\begin{equation*}
MT(d-r+k) \to MT(d-r+k+1) \to MT(d-r+k+1,1).
\end{equation*}
The corresponding long exact sequences of homotopy groups form a spectral sequence with
\begin{equation*}
E^1_{s,t} = \pi_{s+t}(MT(s,1))\cong \pi_t^s(S^0)\oplus \pi_t(\Sigma^\infty B(d))
\end{equation*}
for $d-r< s\leq d$ and $E^1_{s,t}=0$ otherwise, and with differentials
\begin{equation*}
d^k: E^k_{s,t} \to E^k_{s-k,t+k-1}.
\end{equation*}

There is also a filtration
\begin{equation}\label{filt2}
MT(d-r,0) \to MT(d-r+1,1) \to \dotsm \to MT(d-1,r-1) \to MT(d,r).
\end{equation}
The obvious map from \eqref{filt} to~\eqref{filt2} induces a homotopy equivalence of cofibers. Hence it induces an isomorphism of the associated spectral sequences.

\begin{prop}
The spectral sequence converges to $\pi_{s+t}(MT(d,r))$ filtered by the subgroups
\begin{equation*}
F_{s,t}=\Ima(\pi_{s+t}(MT(s,r-d+s))\to\pi_{s+t}(MT(d,r)))
\end{equation*}
such that $E^\infty_{s,t}=F_{s,t}/F_{s-1,t+1}$. 
\end{prop}

\begin{proof}
Using the construction \eqref{filt2}, the result follows from standard convergence theorems since $\pi_*(MT(d-r,0))=0$, see e.g.\ \cite{hatcher2}, Proposition 1.2. 
\end{proof}

The first differential $d^1: E^1_{s,t} \to E^1_{s-1,t}$ in the spectral sequence is induced by the composite map $\tau$
\begin{equation}\label{q}
\xymatrix{{}&{\Sigma MT(d-s-1)}\ar[r]^q\ar[d]^{\simeq} & {\Sigma MT(d-s-1,1)}\\
{MT(d-s,1)}\ar[r]^{\partial} & {\Sigma p_{V_1}^*MT(d-s)} & {.}
}
\end{equation}
This is also the boundary map in the cofibration sequence
\begin{equation}\label{seq}
MT(d-s-1,1) \to MT(d-s,2) \to MT(d-s,1) \xrightarrow{\tau} \Sigma MT(d-s-1,1).
\end{equation}
It turns out that this is a familiar map:
\begin{prop}
The map $\tau :MT(d,1)\to \Sigma MT(d-1,1)$
is the Becker--Gottlieb transfer associated to the sphere bundle
\begin{equation*}
S^{d-1} \to V_1(U_{d}) \xrightarrow{p_{V_1}} B(d).
\end{equation*}
\end{prop}

\begin{proof}
Recall how the Becker--Gottlieb transfer
\begin{equation*}
\Sigma^{n+d}G(d,n) \to \Sigma^{n+d}V_1(U_{d,n})
\end{equation*}
is defined. We may think of $V_1(U_{d,n})$ as a subset of $U_{d,n}$. This extends to an embedding of the normal bundle $\R \times V_1(U_{d,n}) \to U_{d,n}$ in the obvious way.
Hence there is a map
\begin{equation}\label{tau1}
\gamma_1 : \Th(U_{d,n}) \to \Th(\R \times V_1(U_{d,n}) \to V_1(U_{d,n}))
\end{equation}
given by collapsing $G(d,n)$.
Let $\gamma_2 : \R \times V_1(U_{d,n}) \to p_{V_1}^* U_{d,n}$ be the inclusion of a subbundle over $V_1(U_{d,n})$ that takes a point $(t,v\in P)$ where $t\in \R$ and $P\in G(d,n)$ to  $tv$ in the fiber over $v\in P$. 
The transfer is then defined to be the composition
\begin{equation*}
\Th(U_{d,n}^\perp \oplus U_{d,n})  \xrightarrow{\gamma_1 \oplus i} \Th(\R \oplus p_{V_1}^*U_{d,n}^\perp)\xrightarrow{\gamma_2 \oplus 1} \Th(p_{V_1}^*U_{d,n} \oplus p_{V_1}^*U_{d,n}^\perp)
\end{equation*}
where $i$ is the natural inclusion of fibers.

Recall that
\begin{equation}\label{MTd1}
MT(d,1)_n  = \Th(p_{V_1}^*U_{d,n}^\perp)/ \Th(p_{W_1}^*U_{d,n}^\perp ) =\Th(U_{d,n}^\perp \oplus U_{d,n}) =\Sigma^{n+d}(G(d,n)_+),
\end{equation}
and the boundary map $\partial : MT(d,1)_n \to \Sigma MT(d-1)_{n}$ is given by collapsing the subset $MT(d)_n = \Th(U_{d,n}^\perp\to G(d,n))$.
This is exactly what the map $\gamma_1\oplus i$ does.

Thus we just need to see that $\gamma_2 \oplus 1$ is homotopic to the map $q$ in \eqref{q}.
This $q$ can be thought of as the inclusion 
\begin{equation*}
\Sigma MT(d-1)_n = \Th(\R \oplus U_{d-1,n}^\perp ) \to  \Th(\R \oplus U_{d-1,n}^\perp\oplus U_{d-1,n})=\Sigma MT(d-1,1)_n,
\end{equation*}
while $\gamma_2 \oplus 1$ was the inclusion $\Th(\R \oplus p_{V_1}^*U_{d,n}^\perp  ) \to \Th(p_{V_1}^*(U_{d,n}^\perp \oplus U_{d,n}))$. The result now follows by commutativity of the diagram
\begin{equation*}
\xymatrix{{\Th(\R \oplus U_{d-1,n}^\perp )}\ar[r]^-{q}\ar[d]&{\Th(\R \oplus U_{d-1,n}^\perp\oplus U_{d-1,n})}\ar[d]\\
{\Th(\R \oplus p_{V_1}^*U_{d,n}^\perp ) }\ar[r]^-{\gamma_2 \oplus 1}&{\Th(p_{V_1}^*(U_{d,n}^\perp\oplus U_{d,n}))}
}
\end{equation*}
where the vertical maps are the homotopy equivalences of spectra induced by the inclusion $G(d-1,n)\to V_1(U_{d,n})$.
\end{proof}

\begin{cor}\label{bg}
The composition
\begin{equation*}
\Sigma^\infty S^0 \to \Sigma^\infty B(d)_+ \xrightarrow{\tau} \Sigma^\infty B(d-1)_+ \to \Sigma^\infty S^0
\end{equation*}
has degree $\chi(S^d)$, which is $2$ for $d$ even and $0$ for $d$ odd. Thus on the $\pi_*(S^0)$ summands of $E^1,$ $d^1$ is multiplication by $\chi(S^d)$.
\end{cor}

\begin{proof}
This follows from \cite{becker}, Property (3.2) and (3.4).
\end{proof}

The computations of the groups $\pi_q(MTSO(d,r))$ from \cite{adb} yields the limit of the spectral sequence in a range. This allows us to determine all differentials in the spectral sequence for $MTSO(d,r)$ that enter the first four rows, at least in the stable area $ t+s < 2(d-r)$. This is displayed in Fi\-gure~\ref{spfl2}. An arrow indicates a non-zero differential. The horizontal differentials are multiplication by $2$ according to Corollary \ref{bg}.
The picture is repeated horizontally with a period of 4. The spectral sequence does not seem to give any new information about $\pi_*(MTSO(d,r))$.

The periodicity map
\begin{equation*}
 MT(d,r) \to \Sigma^{-ka_r} MT(d+ka_r,r) 
\end{equation*}
induces an isomorphism of spectral sequences in the stable area $t+s < 2(d-r)$. Thus it is tempting to extend the spectral sequence to a half plane spectral sequence by replacing all terms on the $E^1$-page by their stable versions and defining differentials by periodicity. We do this more formally in the next section, allowing us to determine the limit.

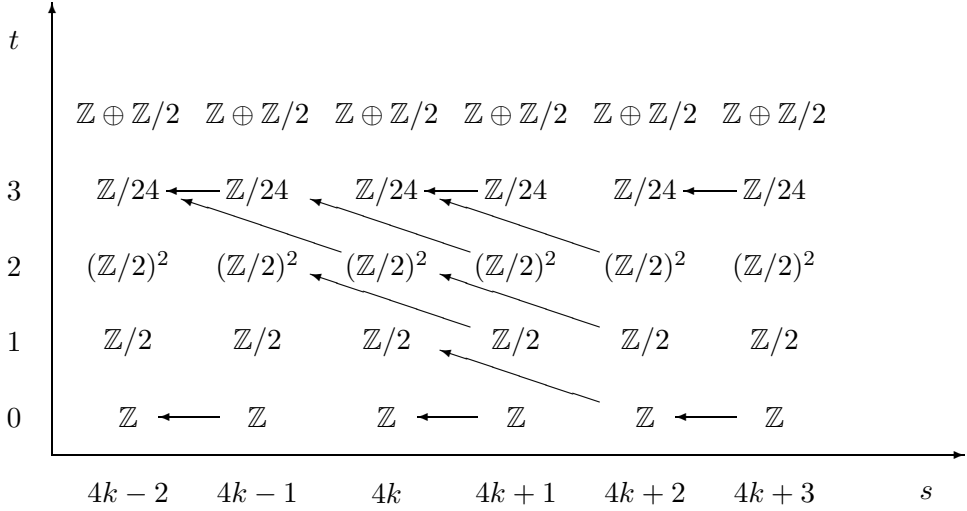
\begin{figure}
\begin{equation*}
\setlength{\unitlength}{1cm}
\begin{picture}(13,7)
\put(0.5,0.5){\vector(0,1){6}}
\put(0.5,0.5){\vector(1,0){12}}
\put(1.5,1){\makebox(0,0){$\Z$}}
\put(3.2,1){\makebox(0,0){$\Z$}}
\put(4.9,1){\makebox(0,0){$\Z$}}
\put(6.6,1){\makebox(0,0){$\Z$}}
\put(8.3,1){\makebox(0,0){$\Z$}}
\put(10,1){\makebox(0,0){$\Z$}}

\put(1.5,2){\makebox(0,0){$\Z/2$}}
\put(3.2,2){\makebox(0,0){$\Z/2$}}
\put(4.9,2){\makebox(0,0){$\Z/2$}}
\put(6.6,2){\makebox(0,0){$\Z/2$}}
\put(8.3,2){\makebox(0,0){$\Z/2$}}
\put(10,2){\makebox(0,0){$\Z/2$}}

\put(1.5,3){\makebox(0,0){$(\Z/2)^2$}}
\put(3.2,3){\makebox(0,0){$(\Z/2)^2$}}
\put(4.9,3){\makebox(0,0){$(\Z/2)^2$}}
\put(6.6,3){\makebox(0,0){$(\Z/2)^2$}}
\put(8.3,3){\makebox(0,0){$(\Z/2)^2$}}
\put(10,3){\makebox(0,0){$(\Z/2)^2$}}

\put(1.5,4){\makebox(0,0){$\Z/24$}}
\put(3.2,4){\makebox(0,0){$\Z/24$}}
\put(4.9,4){\makebox(0,0){$\Z/24$}}
\put(6.6,4){\makebox(0,0){$\Z/24$}}
\put(8.3,4){\makebox(0,0){$\Z/24$}}
\put(10,4){\makebox(0,0){$\Z/24$}}

\put(1.5,5){\makebox(0,0){$\Z \oplus \Z/2$}}
\put(3.2,5){\makebox(0,0){$\Z \oplus \Z/2$}}
\put(4.9,5){\makebox(0,0){$\Z \oplus \Z/2$}}
\put(6.6,5){\makebox(0,0){$\Z \oplus \Z/2$}}
\put(8.3,5){\makebox(0,0){$\Z \oplus \Z/2$}}
\put(10,5){\makebox(0,0){$\Z \oplus \Z/2$}}

\put(0,1){\makebox(0,0){$0$}}
\put(0,2){\makebox(0,0){$1$}}
\put(0,3){\makebox(0,0){$2$}}
\put(0,4){\makebox(0,0){$3$}}
\put(0,6){\makebox(0,0){$t$}}
\put(1.5,0){\makebox(0,0){$4k-2$}}
\put(3.2,0){\makebox(0,0){$4k-1$}}
\put(4.9,0){\makebox(0,0){$4k$}}
\put(6.6,0){\makebox(0,0){$4k+1$}}
\put(8.3,0){\makebox(0,0){$4k+2$}}
\put(10,0){\makebox(0,0){$4k+3$}}
\put(12,0){\makebox(0,0){$s$}}

\put(2.7,1){\vector(-1,0){0.8}}
\put(2.7,4){\vector(-1,0){0.7}}
\put(6.1,1){\vector(-1,0){0.8}}
\put(6.1,4){\vector(-1,0){0.7}}
\put(9.5,1){\vector(-1,0){0.8}}
\put(9.5,4){\vector(-1,0){0.7}}

\put(7.7,1.2){\vector(-3,1){2.1}}
\put(6.0,2.2){\vector(-3,1){2.1}}
\put(6.0,3.2){\vector(-3,1){2.1}}
\put(7.7,2.2){\vector(-3,1){2.1}}
\put(7.7,3.2){\vector(-3,1){2.1}}
\put(4.3,3.2){\vector(-3,1){2.1}}


\end{picture}
\end{equation*}
\caption{The first differentials in the spectral sequence.}\label{spfl2}
\end{figure}

\subsection{The stable spectral sequence}\label{sss}
In this section, we consider maps $f :X \to BO$ where $X$ is either a compact CW complex, $BO$, $BSO$, or $BSpin$. In all these cases, we use the notation $f^*\MT(d,r)$ etc.\ for the corresponding spectra.

The first step in stabilizing the spectral sequence of the previous section is to replace the spectrum $f^*MT(d,r)$ by $f^*\MT(d,r)$.
 
\begin{lem}\label{flg0}
For $X$ compact and $f:X \to G(d+ma_r,n)$ and $N=ka_{r+l}$, there are long exact sequences
\begin{equation*}
\to \pi_*(f^*MT(d + N,r)) \to \pi_*(f^*MT(d+l+N,r+l)) \to \pi_*(f^*MT(d+l+N,l)) \to 
\end{equation*}
for all $N$ sufficiently large and  $*< 2(d+N-1)-r$.
\end{lem}

\begin{proof}
Let $f^*(U_{d+N,n }) = F$ and $f^*(U_{d+l+N,n })=E$ for simplicity. These are both vector bundles over $X$ and $E \cong F \oplus \R^l$. The maps in the sequence are the maps of Thom spaces over the maps of pairs
\begin{equation*}
(W_r(F),V_r(F))\to (W_{r+l}(E),V_{r+l}(E))
\to (W_l(E),V_l(E)).
\end{equation*}
 We need to see that the map of pairs
\begin{equation}\label{maponVs}
(W_{r+l}(E),V_{r+l}(E) \cup W_r(F)) \to (W_l(E),V_l(E)\cup X\times I)
\end{equation}
is highly connected, since the first pair corresponds to the cofiber of 
\begin{equation*}
f^*MT(d+N,r)\to f^*MT(d+N+l,r+l),
\end{equation*}
while the second pair corresponds to $f^*MT(d+N+l,l)$. A point $(x,s)$ in $ X\times I$ should be interpreted as $(su_1,\dots,su_l )\in W_l(F\oplus \R^l)$ where $(u_1,\dots,u_l)$ is the standard frame in $ 0\oplus \R^l$.  All the spaces in \eqref{maponVs} are fiber bundles over $X$, so it is enough to see that the fibers are highly connected. Now, $W_{d+l +N,r+l}$ and $W_{d+l +N,l}$ are both contractible. The fibers $V_{d+l+N,r+l} \cup W_{d+N,r}$ and $V_{d+l+N,l}$ are $(2(d+N-1)-r)$-connected since the first is the mapping cone of the fiber inclusion $ V_{d+N,r} \to V_{d+l+N,r+l}$ and the other one is the base space.
\end{proof}

\begin{cor}\label{restflg}
There are long exact sequences
\begin{equation}\label{esBO}
\to \pi_*(f^*\MT(d,r)) \to \pi_*(f^*\MT(d+l,r+l)) \to \pi_*(f^*\MT(d+l,l)) \to .
\end{equation}
\end{cor}

\begin{proof}
First consider the compact case.
We may assume $f(X) \subseteq G(d+N,n)$ for all $N=ka_{r+l}$ with $k$ sufficiently large. 
We need to see that the periodicity map defines a map between the long exact sequences in Lemma \ref{flg0}.
The diagram
\begin{equation*}
\xymatrix{{\Sigma^{a_{r+l}} f^*MT(d+N,r)}\ar[r]\ar[d]&{\Sigma^{a_{r+l}} f^*MT(d+l+N,r+l)}\ar[d]\\
{f^*MT(d+N+a_{r+l},r)}\ar[r]&{f^*MT(d+l+N+a_{r+l},r+l)}
}
\end{equation*}
commutes up to homotopy. Thus there is an induced map of the long exact sequences.  It is left to the reader to check that the map of cofibers is actually the periodicity map. The claim follows because direct limits preserve exactness. 

For $X=BO$, $BSO$, or $BSpin$, the proof is similar, except the long exact sequences 
\begin{equation*}
\to \pi_*(f^*MT(d+N,r)) \to \pi_*(f^*MT(d+l+N,r+l)) \to \pi_*(f^*MT(d+l+N,l)) \to,
\end{equation*}
are immediate from Lemma \ref{cofdr}.
\end{proof}

\begin{cor}
The Becker--Gottlieb transfer
\begin{equation*}
\tau_*:\pi_*(\Sigma^\infty B(d)_+) \to \pi_*(\Sigma^\infty B(d-1)_+)
\end{equation*}
depends only on $d $ mod $ 2$ in dimensions $*<d-1$.
\end{cor}

\begin{proof}
The proof of Corollary \ref{restflg} shows that the periodicity map defines a map between the cofibrations \eqref{seq}.
\end{proof}

The sequences \eqref{esBO} fit together to form a  spectral sequence:
\begin{thm} \label{stabss}
There is a spectral sequence converging to $\pi_*(f^*\MT(d,r))$ and with 
\begin{equation*}
E^1_{s,t} = \pi_{s+t}(f^*\MT(s,1)) \cong \pi_{t}^s(X) \oplus \pi_{t}^s(S^0)
\end{equation*}
for $d-r< s \leq d$ and $E^1_{s,t}=0$ otherwise. 
\end{thm}
For $X=BO$ or $BSO$, this is the spectral sequence from Section \ref{thess} but with all groups on the $E^1$-page replaced by their stable versions. 

Replacing all groups in \eqref{esBO} by their $2$-completions yields
\begin{equation*}
\to \pi_*(f^*\MT(d,r))_2^\wedge \to \pi_*(f^*\MT(d+1,r+1))_2^\wedge \to\pi_*(f^*\MT(d+1,1))_2^\wedge \to .
\end{equation*}
This is again an exact sequence because all the groups involved are finitely generated.
Since the inverse limit functor is exact with respect to sequences of profinite groups,
 there are long exact sequences
\begin{equation*}
\to \varprojlim_r \pi_*(f^*\MT(d,r) )_2^\wedge \to \varprojlim_r \pi_*(f^*\MT(d +1,r+1) )_2^\wedge \to  \pi_*(f^*\MT(d+1,1)) _2^\wedge \to.
\end{equation*}
These sequences form a spectral sequence where $\varprojlim_r \pi_*(f^*\MT(d,r) )_2^\wedge$ is filtered by the image of the groups $\varprojlim_r \pi_*(f^*\MT(d-l,r-l))_2^\wedge$ and with
\begin{equation*}
\hat{E}^1_{s,t} = \pi_{s+t}(f^*\MT(s,1))_2^\wedge
\end{equation*}
and differentials
\begin{equation*}
d^k: \hat{E}^k_{s,t} \to \hat{E}^k_{s-k,t+k-1}.
\end{equation*}

\begin{thm}
This spectral sequence $\hat{E}^*_{*,*}$ converges strongly in the sense of \cite{boardman} to $\varprojlim_r\pi_*(f^*\MT(d,r))_2^\wedge$. 
\end{thm}

\begin{proof}
Since $\hat{E}^1_{s,t}=0$ for $t<0$, we have a half-plane spectral sequence with entering differentials in the sense of \cite{boardman}. Thus, by Theorem 7.3 of this paper, it is enough to check that the following three conditions are satisfied:
\begin{itemize}
\item[(i)]
\begin{equation*}
\varprojlim\nolimits_l \varprojlim\nolimits_r \pi_*(f^*\MT(d-l,r-l))_2^\wedge \cong \varprojlim\nolimits_r \varprojlim\nolimits_l \pi_*(f^*\MT(d-l,r-l))_2^\wedge =0,
\end{equation*}
since $\pi_*(f^*\MT(d-l,r-l))$ is zero for $k>r$.
\item[(ii)]
\begin{equation*}
\varprojlim\nolimits_l^1 \varprojlim\nolimits_r \pi_*(f^*\MT(d-l,r-l))_2^\wedge =0.
\end{equation*}
This follows from the fact that 
\begin{equation*}
\varprojlim\nolimits_l \pi_*(f^*\MT(d-l,r-l))_2^\wedge = \varprojlim\nolimits_l^1 \pi_*(f^*\MT(d-l,r-l))_2^\wedge=0
\end{equation*}
and diagram chasing in the diagram defining the double limit.
\item[(iii)]
\begin{equation*}
\varprojlim\nolimits_k^1 \hat{Z}^k_{s,t} = 0.
\end{equation*}
Here $\hat{Z}^k_{s,t}$ denotes the cycles on the $k$th page, i.e.\ the inverse image of 
\begin{equation*}
\Im (\varprojlim\nolimits_r \pi_{s+t-1}(f^*\MT(s-k,r-k))_2^\wedge \to \varprojlim\nolimits_r\pi_{s+t-1}(f^*\MT(s-1
,r))_2^\wedge) 
\end{equation*}
under the map 
\begin{equation*}
\pi_{s+t}(f^*\MT(s,1))_2^\wedge \to \varprojlim_r \pi_{s+t-1}(f^*\MT(s-1,r))_2^\wedge.
\end{equation*}
 But the sequence 
\begin{align*}
 \to  \varprojlim\nolimits_r \pi_{s+t-1}(f^*\MT(s-k,r-k))_2^\wedge \to \varprojlim\nolimits_r &\pi_{s+t-1}(f^*\MT(s-1,r))_2^\wedge \\ 
\to &\pi_{s+t-1}(f^*\MT(s-1,k))_2^\wedge \to 
\end{align*}
is exact, so $\hat{Z}^k_{s,t}$ is actually the kernel of the composite map
\begin{align*}
\pi_{s+t}(f^*\MT(s,1))_2^\wedge &\to \varprojlim\nolimits_r \pi_{s+t-1}(f^*\MT(s-1,r))_2^\wedge \\&\to \pi_{s+t-1}(f^*\MT(s-1,k))_2^\wedge ,
\end{align*}
which is continuous. Thus the $\hat{Z}^k_{s,t}$ are closed subgroups of a 2-profinite group. But $\varprojlim^1$ vanishes for any inverse system of closed subgroups of a profinite group because these are again profinite, the quotients are profinite, and $\varprojlim$ is exact on sequences of profinite groups with continuous maps.
\end{itemize}
\end{proof}

%

The inclusion $\widehat{f^*MT}(d) \to \widehat{f^*MT}(d+1)$ defines a map of these spectral sequences. This is an isomorphism on $\hat{E}^1_{s,t}$ in the area $s+t<d$. Letting $d$ tend to infinity, we obtain a spectral sequence converging to  $\pi_*(\widehat{f^*MT})$.

\begin{thm}
For $X=pt$, $\hat{E}$ from Theorem \ref{stabss}  is the stable EHP spectral sequence.
\end{thm}

\begin{proof}
For a fixed $r$, a homotopy equivalence $\Sigma^{\infty + 1}P_{d,r} \to \V_{d,r}$ as in Theorem \ref{PVeq} coming from a map $P_{d+ka_r,r} \to  V_{d+ka_r,r}$, defines commutative diagrams
\begin{equation}\label{PVhom}
\xymatrix{{\Sigma^{\infty +1} P_{d-k-1,r-k-1} }\ar[r]\ar[d]&{\Sigma^{\infty +1} P_{d-k,r-k}}\ar[r]\ar[d]&{\Sigma^{\infty +1} P_{d-k,1} }\ar[d]\\
{\V_{d-k-1,r-k-1}}\ar[r]&{\V_{d-k,r-k}}\ar[r]&{\V_{d-k,1}.}
}
\end{equation}
This induces an isomorphism between 
 the spectral sequence $\bar{E}$ associated to the filtration 
\begin{equation*}
\Sigma^{\infty + 1} P_{d-r,0}\to \dotsm \to \Sigma^{\infty + 1} P_{d,r}
\end{equation*}
and $\hat{E}$.

If a different Clifford representation is used to define $\Sigma^{\infty + 1}P_{d,r} \to \V_{d,r}$, we have not been able to show that we get a homotopy equivalent diagram \eqref{PVhom}. However, the right vertical map is always a degree $-1$ map. Hence the induced isomorphism $\bar{E}^1_{*,*} \to \hat{E}^1_{*,*}$ is independent of the chosen Clifford map. 
In particular we $r$ tends to infinity, we get an isomorphism of inverse systems of spectral sequences. 

%
Letting $d$ tend to infinity, this is exactly the stable EHP spectral sequence as constructed by Mahowald in~\cite{mahowald}. 
\end{proof}

\begin{rem}
There is a short exact sequence
\begin{equation*}
\varprojlim\nolimits_r^1 \pi_{q+1}(f^*\MT(d,r)_2^\wedge )\to \pi_q(\varprojlim\nolimits_r (f^*\MT(d,r)_2^\wedge)) \to\varprojlim\nolimits_r \pi_q(f^*\MT(d,r)_2^\wedge).
\end{equation*}
In the cases we consider, the $\varprojlim\nolimits^1$ term vanishes for $q<d-1$, see Lemma \ref{endelig} below in the compact case. Hence the spectral sequence converges to 
\begin{equation*}
\pi_{s+t}(\varprojlim\nolimits_r (f^*\MT(d,r)_2^\wedge))=\pi_{s+t}(\widehat{f^*MT}(d){}_2^\wedge)
\end{equation*}
 in the area $s+t < d-1$. Furthermore, $\pi_q(f^*\MT(d,r))\cong \pi_q(f^*\MT(d,r)_2^\wedge) $ for infinitely many $r$, so in fact, the spectral sequence converges to $\pi_{s+t}(\widehat{f^*MT}(d))$ in this area. 
\end{rem} 
The remaining sections are devoted to the study of the homotopy limit $\widehat{f^*MT}(d)$.

\section{The compact case}\label{pbs}
When $X$ is a point, we saw in Theorem \ref{linV} that for $q< d-1$, the map $\hat{\varphi}_0 = \widehat{f^*\varphi}: S^0 \to \widehat{V}_d$ induces
an isomorphism
\begin{equation*}
\widehat{f^*\varphi}_{*}: \pi_q(S^0)_2^\wedge \to \pi_q(\widehat{V}_d)_2^\wedge \cong \pi_q(\widehat{V}_d).
\end{equation*}
The analogue statement holds for the map 
\begin{equation*}
\widehat{f^*\varphi}: f^*MT(d) \to \widehat{f^*MT}(d)
\end{equation*}
for any $f: X \to BO$ where $X$ is compact.
We give a topological argument in Section~\ref{topa}, using induction on the number of cells in $X$ and starting with Lin's theorem. In Section \ref{sing}, we outline an algebraic approach to determine $\widehat{f^*MT}(d)$. Finally, in Section \ref{singerc}, we recall the Singer construction and see how this applies to give an alternative proof.

\subsection{Topological approach}\label{topa}
We first need a few lemmas:
\begin{lem}\label{homiso}
If $g:X \to Y$ is a homotopy equivalence and the diagram
\begin{equation*}
\xymatrix{{X}\ar[r]^-{f_X}\ar[d]^-{g}&{BO}\\
{Y}\ar[ur]_-{f_Y}&{}
}
\end{equation*}
commutes, then the induced map
\begin{equation*}
g_*: \varprojlim_r \pi_*( f^*_X\MT(d,r)) \to \varprojlim_r \pi_*(f^*_Y\MT(d,r))
\end{equation*}
is an isomorphism.
\end{lem}

\begin{proof}
Since $g_* : f_X^*MT(d +ka_r,r)\to f_Y^*MT(d+ka_r,r)$ is a homotopy equivalence, it induces an isomorphism on homotopy groups. Thus it also induces an isomorphism in the limit.
\end{proof}

\begin{lem} \label{endelig}
Assume $q<d$ and let $f: X \to BO$ be defined on a finite CW complex $X$. Then $\pi_q(f^* \MT(d,r))$ is a finite $2$-primary group when $d-r$ is odd, and thus $\varprojlim_r \pi_q(f^*\MT(d,r))$ is $2$-profinite. For $d$ odd, this also holds when $d=q$. 
Moreover, the periodicity map
\begin{equation*}
\Sigma^{ka_r} f^*MT(d,r) \to f^*MT(d + ka_r,r)
\end{equation*}
is a $(2(d-r) + ka_r+1)$-equivalence. 
\end{lem}

\begin{proof}
In the case where $X$ is a point, i.e.\ $f^*\MT(d,r) =\V_{d,r}$, the first statement is true by Remark \ref{improveV}.

Now suppose that $X$ is obtained from $Y$ by glueing on an $n$-cell $D^n$ such that $Y\cap D^n = S^{n-1}$. Then there is a Mayer--Vietoris sequence
\begin{align}\label{MV0}
 \to \pi_q( f^*MT(d,r)_{\mid S^{n-1}})\to \pi_q(f^*&MT(d,r)_{\mid D^n})\oplus \pi_q( f^*MT(d,r)_{\mid Y})\\
 \to &\pi_q( f^* MT(d,r)_{\mid X})\to  \pi_{q-1}( f^*MT(d,r)_{\mid S^{n-1}})\to \nonumber.
\end{align} 
The periodicity map defines a map of these exact sequences, and since direct limits preserve exactness, there is also an exact sequence
\begin{align*}
\to \pi_q( f^*\MT(d,r)_{\mid S^{n-1}})\to \pi_q(f^*&\MT(d,r)_{\mid D^n})\oplus \pi_q(f^* \MT(d,r)_{\mid Y}) \\
\to &\pi_q( f^*\MT(d,r)_{\mid X})\to \pi_{q-1}( f^*\MT(d,r)_{\mid S^{n-1}}) \to .
\end{align*} 
Since $D^n \simeq pt$, the theorem is known for $X=D^n$ by Lemma \ref{homiso}.

First consider $X= S^{n}$  as the union of two disks glued together along a copy of $S^{n-1}$. By induction on $n$, $\pi_q( f^*\MT(d,r)_{\mid S^{n}})$ must be a finite 2-primary group in order to fit into the exact sequence. 
 
For a general simplicial complex $X$, it now follows by induction on the number of cells in $X$ that $\pi_q(f^* \MT(d,r)_{\mid X})$ is a finite 2-primary group when $d-r$ is odd, by applying the Mayer--Vietoris sequence and the sphere case. Finally we get the result for a general CW complex from Lemma \ref{homiso}.

The last statement is true for a point and the general claim follows by an induction argument similar to the above applied to the periodicity map between the sequences \eqref{MV0}.
\end{proof}

\begin{lem}\label{limMV}
Let $f: X \to BO$ be given and assume that $X$ is the union of two finite subcomplexes $Y_1$ and $Y_2$. There is an exact Mayer--Vietoris sequence
\begin{align*}
 \varprojlim_r\pi_q( f^*\MT(d,r)_{\mid Y_1 \cap Y_2})_2^\wedge\to \varprojlim_r\pi_q(f^*\MT(d,r&)_{\mid Y_1})_2^\wedge\oplus \varprojlim_r\pi_q(f^* \MT(d,r)_{\mid Y_2})_2^\wedge \\
\to \varprojlim_r\pi_q( f^*\MT(d,r)_{\mid Y_1 \cup Y_2})_2^\wedge \to &\varprojlim_r\pi_{q-1}( f^*\MT(d,r)_{\mid Y_1 \cap Y_2})_2^\wedge .
\end{align*}
\end{lem}

\begin{proof}
As in the proof of Lemma \ref{endelig}, there is a Mayer--Vietoris sequence
\begin{align*}
\to \pi_q(f^* \MT(d,r)_{\mid Y_1 \cap Y_2}) \to \pi_q(f^*\MT(d,r)_{\mid Y_1}) \oplus \pi_q&( f^*\MT(d,r)_{\mid Y_2})\\
 &\to \pi_q(f^* \MT(d,r)_{\mid Y_1 \cup Y_2}) \to .
\end{align*}
Since all groups involved are finitely generated, replacing them by their 2-completions yields a new long exact sequence. The inverse limit functor is exact with respect to long exact sequences of profinite groups, so taking the inverse limit over $r$ yields the desired sequence.
\end{proof}

%

We are now ready to prove the main theorem of this section. Below, $B(d)$ may denote either $BO(d)$, $BSO(d)$, or $BSpin(d)$, and $MT(d)$ denotes the corresponding spectrum with $MT = \varinjlim_d MT(d)$. 
\begin{thm}\label{hovedthm}
Let $f: X \to BO(d)$ with $X$ a finite complex. Then 
\begin{equation*}
\widehat{f^* \varphi}_{*} : \pi_q(f^*MT(d))_2^\wedge \to \varprojlim_r \pi_q(f^*\MT(d,r))_2^\wedge  
\end{equation*}
is a $(d-1)$-equivalence.
In particular for $X\subseteq B(d)$ compact, we get 
\begin{equation*}
\varinjlim_{X \subseteq B(d)} \pi_q(\widehat{MT}(d)_{\mid X}) \cong \pi_q(MT(d))_2^{\wedge}
\end{equation*}
for $q<d-1$. Taking the direct limit over $d$, 
\begin{equation*}
\varinjlim_d \varinjlim_{X \subseteq B(d)} \pi_q(\widehat{MT}_{\mid X}) \cong {\pi_q(MT)}_2^\wedge
\end{equation*}
for all $q$.
\end{thm}

\begin{proof}
We want to do an induction on the cells in $X$ as in the proof of Lemma \ref{endelig}. Assume that $X$ is built from $Y$ by glueing on a disk $D^n$ such that $D^n \cap Y = S^{n-1}$.
The map $\widehat{f^*\varphi}_*$ takes the Mayer--Vietoris sequence
\begin{equation*}
\to \pi_q(MT(d)_{\mid S^{n-1}})_2^\wedge \to \pi_q(f^*MT(d)_{\mid D^n})_2^\wedge\oplus \pi_q(f^*MT(d)_{\mid Y})_2^\wedge \to \pi_q( \widehat{f^*MT}(d)_{\mid X })_2^\wedge \to 
\end{equation*}
to the exact sequence of Lemma \ref{limMV}. The theorem is known for $X=pt$, so the induction proceeds as in the proof of Lemma \ref{endelig}.
\end{proof}


\subsection{Algebraic approach}\label{sing}

From now on we shall only consider cohomology with $\Z/2$ coefficients understood. The mod 2 Steenrod algebra will  be denoted by $\A$.
In \cite{rognes}, Proposition~2.2, the following version of the Adams spectral sequence for an inverse system 
\begin{equation*}
\dotsm \to Y_{r+1} \to Y_r
\end{equation*}
of spectra is constructed:

\begin{thm}\label{invadams}
Assume that the spectra $Y_r$ have finite $\Z/2$ cohomology in each dimension and each $\pi_*(Y_r)$ is bounded below. Then there is a spectral sequence with $E_2$-term $E_2^{s,t} = \Ext_\A^{s,t}(\varinjlim_r H^*(Y_r),\Z/2)$ converging strongly to the  homotopy groups $\pi_{t-s}((\varprojlim_r Y_r)_2^\wedge)$. 

For each $k$, $E_k^{s,t}$ is the inverse limit of the Adams spectral sequences $E_k^{s,t}(Y_r)$ for the $Y_r$.
\end{thm} 

We are interested in the case $Y_r = f^*\MT(d,r)$ and $(\varprojlim_r Y_r)_2^\wedge = \widehat{f^*MT}(d)_2^\wedge$.
Thus we must investigate $\varinjlim_r H^*(f^*\MT(d,r))$ in order to apply the theorem. 


For simplicity, we first introduce some notation.
\begin{defi}\label{defWH}
Let $B(d)$ denote either $BO(d)$, $BSO(d)$, or $BSpin(d)$. Let $MT$ denote the corresponding spectra. Define
\begin{align}\nonumber
H(d)^*=& H^*(B(d))\\ \nonumber
H^*=& H^*(\varinjlim_d B(d)) = \varprojlim_d H^*(B(d))\\ \label{navn}
WH(d)^*=& \varinjlim _r H^*(\MT(d,r))\\ \nonumber
WH^*=& \varprojlim_d WH(d)^*.
\end{align}
For $X$ compact and $f:X\to BO$, let
\begin{align*}
WH(X)(d)^* &= \varinjlim_r H^*(f^* \MT(d,r))\\
WH(X)^* &= \varprojlim_d\varinjlim_r H^*(f^* \MT(d,r))
\end{align*}
\end{defi}

Recall that 
\begin{align*}
H^*(BO)&\cong \Z/2[w_1,w_2, \dots ]\\
H^*(BSO) &\cong \Z/2[w_2,w_3, \dots ]\\
H^*(BSpin) &\cong H^*(BSO)/\A w_2
\end{align*}
where $\Z/2[x_i,\, i\in I ]$ denotes the polynomial algebra on generators $x_i$ in dimension $i$ and $w_i$ is the $i$th Stiefel--Whitney class for the universal bundle. Multiplication by a Thom class $\bar{u}$ for the complement of the universal bundle defines an isomorphism from $H^*(B(d))$ to $H^*(MT(d))$. We shall often omit the $\bar{u}$ from the notation for $H^*(MT(d))$.
\begin{thm}\label{coholim}
$WH(d)^*$ is isomorphic as an $H^*$-module to
\begin{equation*}
H^* \otimes \Z/2\{\tilde{w}_l,l\leq d\}
\end{equation*}
where $\Z_2\{\tilde{w}_l,l\leq d\}$ is the graded vector space with basis $\tilde{w}_l$ in dimension $l$. The Steenrod algebra $\mathcal{A}$ acts by the Cartan formula. The action on $H^*$ is the usual one, while the action on $\tilde{w}_l$ is given by the formula
\begin{equation}\label{Awl}
\Sq^k(\tilde{w}_l)= \sum_{j=0}^k \sum_{i=0}^j \binom{j-l}{i}w_{j-i}\bar{w}_{k-j}\tilde{w}_{l+i}
\end{equation}
where $\bar{w}_i$ are the dual Stiefel--Whitney classes characterized by $\bar{w}_0=1$ and
\begin{equation*}
\sum_{j=1}^i w_j\bar{w}_{i-j}=0.
\end{equation*}
\end{thm}

\begin{proof}
It was shown in \cite{adb} that $H^*(MT(d,r))$ is the $H^*$-ideal in $H^*(MT(d))$ ge\-ne\-ra\-ted by $w_{d-r+1},\dots ,w_{d}$.
The periodicity map $g: MT(d,r) \to \Sigma^{-a_r}MT(d+a_r,r) $ induces an isomorphism
\begin{equation}\label{percoho}
g^*:  H^{*}(\Sigma^{-a_r}MT(d+a_r,r))\to H^*(MT(d,r))
\end{equation}
 in dimensions $* < 2(d-r)$. Thus, the inverse system of cohomology groups
\begin{equation*}
\dotsm \to H^{*}(\Sigma^{-(k+1)a_r}MT(d+(k+1)a_r,r)) \to H^*(\Sigma^{-ka_r}MT(d+ka_r ,r))
\end{equation*}
stabilizes in each dimension and therefore
\begin{equation*}
H^*(\MT(d,r))\cong \varprojlim_k H^{*}(\Sigma^{-ka_r}MT(d+ka_r,r)).
\end{equation*}

It was also shown that the map $g^*$ takes the generators $w_{d+a_r-r+1},\dots ,w_{d+a_r}$ to $w_{d-r+1},\dots ,w_{d}$ and commutes with the $H^*(d+a_r)$-action. 
Furthermore, both are isomorphic to the free $H^*(d+a_r)$-module on these generators up to dimension $2(d-r)$.

Thus, in the limit the cohomology groups become
\begin{equation*}
H^*(\MT(d,r)) \cong H^* \otimes \Z_2\{\tilde{w}_l,d-r+1 \leq l\leq d\}
\end{equation*}
where $\tilde{w}_l$ corresponds to $w_{l+ka_r} \in H^l(\Sigma^{-ka_r}MT(d+ka_r,r))$.
Taking the direct limit over $r$ proves the claim.

The formula for the $\A$-action on $\tilde{w}_l$ follows from the formula
\begin{equation}\label{binom}
\Sq^k(w_m)=\sum_{l=0}^k \binom{m-k+l-1}{l}w_{k-l}w_{m+l}
\end{equation}
in $H^*(BO)$, see \cite{milnors}, Problem 8-A, and the fact that $\Sq^i(\bar{u})=\bar{w}_i$. 
\end{proof}


\begin{lem}\label{grassmann}
$WH(X)^*$ is isomorphic to $H^*(X)\otimes \Z/2\{\tilde{v}_l,\, l\in \Z \}$ as an $H^*(X)$-module.
The map $H^*(\mathcal{MTO}(d,r)) \to H^*(f^*\MT(d,r))$ takes $\tilde{w}_l$ to $\tilde{v}_l$ and is a map of $H^*(BO)$-modules via the map $f^*: H^*(BO) \to H^*(X)$. 
\end{lem}

\begin{proof}
Look at the Serre spectral sequence for the fibration 
\begin{equation*}
V_{d+ka_r,r} \to V_r(f^*U_{d+ka_r,n}) \to X
\end{equation*}
for $n$ and $k$ large. 
This has $E_2$-term
\begin{equation*}
E_2^{p,q} \cong H^p(X)\otimes H^q(V_{d+ka_r,r}).
\end{equation*}
Comparing with the spectral sequence for $V_r(U_{d+ka_r,n})\to G(d+ka_r,n)$ and using the multiplicative structure, we see that  there can be no non-trivial differentials in the lower left corner of the spectral sequence. Hence
\begin{equation*}
H^*(V_r(f^*U_{d+ka_r,n})) \cong H^*(X)\otimes  {H}^*(V_{d+ka_r,r}) 
\end{equation*}
for $*<2(d-r) + ka_r$ as $H^*(X)$-modules.

In the exact sequence
\begin{equation*}
H^*(W_r(f^*U_{d+ka_r,n}),V_r(f^*U_{d+ka_r,n}))\to{H^*(W_r(f^*U_{d+ka_r,n}))}\to {H^*(V_r(f^*U_{d+ka_r,n}))},
\end{equation*}
the last map is an injection by the spectral sequence. Thus for $*<2(d-r) + ka_r$,
\begin{equation*}
H^*(W_r(f^*U_{d+ka_r,n}),V_r(f^*U_{d+ka_r,n})) \cong H^*(X) \otimes {H}^{>0}(\Sigma V_{d+ka_r,r}).
\end{equation*}

To see that the map $H^*(\mathcal{MTO}(d,r)) \to H^*(f^*\MT(d,r))$ is  as claimed, consider
\begin{equation*}
\xymatrix{{H^*(W_r(f^*U_{d+ka_r,n}),V_r(f^*U_{d+ka_r,n}))}\ar[r]&{H^*(\Sigma V_{d+ka_r,r})}\\
{H^*(W_r(U_{d+ka_r,n}), V_r(U_{d+ka_r,n}))} \ar[u]\ar[r]&{H^*(\Sigma V_{d+ka_r,r}).}\ar[u]^{\cong}
}
\end{equation*}
The vertical map to the left takes the generators $w_{d+ka_r-r+1},\dots , w_{d+ka_r}$  to the generators $v_{d+ka_r-r+1},\dots , v_{d+ka_r}$ of $H^0(X) \otimes {H}^{>0}(V_{d+ka_r,r})$, which can be proved by an induction on $r$ using the right hand side of the diagram. It commutes with the $H^*(G(d+ka_r,n))$-module structure. Applying a Thom isomorphism and letting $k$ tend to infinity, the claim follows.
\end{proof}

\subsection{The Singer construction}\label{singerc}
Let $\Z/2[t, t^{-1}]$ be the graded ring of Laurent polynomials in the variable $t$. This is an $\A$-module with $\Sq^k(t^l) = \binom{l}{k} t^{k+l}$. For $l$ negative, the binomial coefficient should be interpreted as
\begin{equation*}
\binom{l}{k} = \frac{l\cdot(l-1)\dotsm (l-k+1)}{k!}.
\end{equation*}
As $\A$-modules,
\begin{equation*}
\varprojlim_d \varinjlim_r H^*(\V_{d,r}) \cong \Sigma \Z/2[t, t^{-1}].
\end{equation*}

Lin's theorem is based on the observation that the map 
\begin{equation}\label{epsmap}
\Sigma \Z/2[t, t^{-1}]\to \Z/2
\end{equation}
induced by $\hat{\varphi}_0 : S^0 \to \widehat{V}$  is an isomorphism on $\Ext$-groups 
\begin{equation*}
\Ext^{s,t}_\A(\Z/2,\Z/2) \cong \Ext^{s,t}_\A(\Sigma \Z/2[t, t^{-1}],\Z/2).
\end{equation*} 
See e.g.\ \cite{adams} for a proof.

The map \eqref{epsmap} was later generalized by Singer in \cite{singer} as follows. Let $M$ be an $\A$-module. Then the Singer construction $R_+M$ is the graded vector space $\Sigma (\Z/2[t, t^{-1}] \otimes M)$, but not with the Cartan action of $\A$. Rather, this is given by the formula
\begin{equation}\label{sqsinger}
\Sq^a(t^b \otimes x) = \sum_j \binom{b-j}{a-2j} t^{a+b-j} \otimes \Sq^j(x).
\end{equation}
The advantage of this module structure is that it makes the map $\epsilon : R_+M \to M$ given by $t^k \otimes x \mapsto \Sq^{k+1}(x)$ into an $\A$-homomorphism.

It was proved by Gunawardena and Miller in \cite{miller} that $\epsilon $ is a $\Tor$-equivalence. In particular it induces isomorphisms
\begin{equation*}
\Ext^{s,t}_\A(M,\Z/2) \to \Ext^{s,t}_\A(R_+M,\Z/2).
\end{equation*}

The inclusion $\V_{d,r} \to f^*\MT(d,r)$ induces an $\A$-linear projection
\begin{equation*}
WH(X)^* = H^*(X) \otimes \Z/2\{\tilde{v}_l,\, l\in \Z\} \to \Sigma \Z/2[t,t^{-1}]
\end{equation*}
with kernel $H^{>0}(X) \otimes \Z/2\{\tilde{v}_l,\, l\in \Z\}$.
Thus $WH^*$ looks like the Singer construction applied to $H^*(f^*MT)$. We shall see that this is indeed the case when $X$ is compact.

\begin{defi}\label{tmult}
Let $\Z/2[t,t^{-1}]$ act on $WH(X)^*$ by the formula $t^l(x\tilde{v}_k) = x\tilde{v}_{k+l}$ for $x\in H^*(X)$. For any module  $M$, $\Z/2[t,t^{-1}]$ acts on $R_+(M) = \Sigma(\Z/2[t,t^{-1}] \otimes M)$ in the obvious way.
\end{defi} 

\begin{thm}\label{singer}
There is a $\Z/2$-linear map $\Phi$ such that the diagram 
\begin{equation}\label{Phidef}
\vcenter{\xymatrix{{WH(X)^*}\ar[r]^-{\Phi}\ar[dr]_{(\widehat{f^*\varphi})^*}&{R_+(H^*(f^*MT))}\ar[d]^{\epsilon}\\
{}&{H^*(f^*MT)}
}}
\end{equation} 
commutes and $\Phi$ commutes with the $\Z/2[t,t^{-1}]$-actions.
\end{thm}

\begin{proof}
We must construct a map of the form $\Phi(x\tilde{v}_l)= t^{l-1}\phi(x)$ where $\phi(x)$ is linear of the form $\sum_i  t^{-i} \otimes \phi_i(x)$ with $\phi_i(x) \in H^*(f^*MT)$. Then it will automatically commute with the $\Z/2[t,t^{-1}]$-actions. For the diagram \eqref{Phidef} to commute, these $\phi_i$ must satisfy 
\begin{equation*}
xw_l = (\widehat{f^*\varphi})^*(x\tilde{v}_l) =\epsilon \Phi(x\tilde{v}_l) =\sum_{i=0}^N Sq^{l-i}(\phi_i(x))
\end{equation*}
for all $l$. Here $w_l$ is the Stiefel--Whitney class for the pull-back $f^*U_d$ for $d$ large. This can be written as a matrix equation:
\begin{equation}\label{matrix}
\begin{pmatrix} 1 & \Sq^1 & \dotsm &\Sq^l \\
 0 & 1 & {} & \Sq^{l-1} \\
 \vdots & \ddots & \ddots & \vdots\\
 0 & \dots & 0 & 1 
\end{pmatrix}
\begin{pmatrix} 
\phi_l(x)  \\
 \vdots  \\
 \phi_1(x) \\
 \phi_0(x)
\end{pmatrix}
=
\begin{pmatrix} 
xw_l  \\
 \vdots  \\
 xw_1 \\
 x  
\end{pmatrix}
\end{equation}
Let $\chi(\Sq^k)\in \A$ denote the dual squares defined inductively by $\chi(\Sq^0)=1$ and 
\begin{equation*}
\sum_i \chi(\Sq^i)\Sq^{k-i} = 0.
\end{equation*}
 Then the matrix
\begin{equation}\label{chimatrix}
\begin{pmatrix} 1 & \chi(\Sq^1) & \dotsm &\chi(\Sq^l) \\
 0 & 1 & {} & \chi(\Sq^{l-1}) \\
 \vdots & \ddots & \ddots & \vdots\\
 0 & \dots & 0 & 1 
\end{pmatrix}
\end{equation}
is a right inverse for the matrix in equation \eqref{matrix}, and multiplication on both sides yields a formula defining the $\phi_i$. Note that only finitely many $\phi_i$ can be non-zero, since $H^*(f^*MT)$ is finite and $\phi_i$ has degree $i+\deg(x)$. Clearly, $\Phi$ is linear because each $\phi_i$ is.
\end{proof}

\begin{lem}\label{sqt}
The formula 
\begin{equation*}
\Sq^k(tx)=t\Sq^k(x) + t^2\Sq^{k-1}(x)
\end{equation*}
holds in both $WH(X)^*$ and $R_+(H^*(f^*MT))$.
\end{lem}

\begin{proof}
The formulas hold in $WH^*$ and $R_+(H^*(f^*MT))$. This is a straightforward check using the formulas \eqref{sqsinger} and \eqref{Awl}. In $WH^*(X)\cong H^*(X)\otimes \{\tilde{v}_l,\, l\in \Z\}$, note that it is enough to check the formulas on $\tilde{v}_l$. For these, the claim follows from the formulas in $WH^*$ by considering the map $WH^* \to WH(X)^*$.
\end{proof}

\begin{thm}\label{extiso}
The map $\Phi$ in Theorem \ref{singer} is an isomorphism of Steenrod modules. In particular, $(\widehat{f^*\varphi})^*$ induces an isomorphism on Ext-groups.
\end{thm}

\begin{proof}
First we show that $\Phi$ is an $\A$-homomorphism.
Assume that for some $k$, 
\begin{equation*}
\Phi(\Sq^{k-1}(x)) = \Sq^{k-1}(\Phi(x))
\end{equation*}
for all $x$. It is clearly satisfied for $k=1$. Then by Lemma \ref{sqt}
\begin{align*}
\Phi(\Sq^{k}(tx)) =& t\Phi(\Sq^k(x)) + t^2\Phi(\Sq^{k-1}(x))\\
 =& t\Phi(\Sq^k(x)) + t^2\Sq^{k-1}(\Phi(x))
\end{align*}
and
\begin{equation*}
\Sq^k(\Phi(tx))=t\Sq^k(\Phi(x)) + t^2\Sq^{k-1}(\Phi(x))
\end{equation*}
since $\Phi$ commutes with $t$.
Introducing the notation 
\begin{equation*}
\delta(x)=\Phi(\Sq^k(x))-\Sq^k(\Phi(x)),
\end{equation*}
the above implies that
\begin{equation*}
\delta(tx) = t\delta(x).
\end{equation*}
Iterating this yields $\delta(t^lx) = t^l\delta(x)$ for all $l \in \Z$. 

We must show that $\delta(x)=0$.
Write $\delta(x) = \sum_{i=-N}^N t^i \otimes \delta_i$ and note that
\begin{align*}
\epsilon(\delta(x))&= \epsilon(\Phi(\Sq^k(x)))- \epsilon(\Sq^k(\Phi(x)))\\
 &= (\widehat{f^*\varphi})^*(\Sq^k(x))-\Sq^k(\widehat{f^*\varphi})^*(x))\\
  &=0
\end{align*}
by commutativity of \eqref{Phidef}.
This means that for all $k$,
\begin{equation*}
\epsilon(\delta(t^kx)) = \epsilon(\sum_{i=-N}^N t^{i+k} \otimes \delta_i) = \sum_{i=-N}^N \Sq^{i+k+1}(\delta_i)=0.
\end{equation*}
This yields the following matrix equation:
\begin{equation*} 
\begin{pmatrix} 1 & \Sq^1 & \dotsm &\Sq^{2N} \\
 0 & 1 & {} & \Sq^{2N-1} \\
 \vdots & \ddots & \ddots & \vdots\\
 0 & \dots & 0 & 1 
\end{pmatrix}
\begin{pmatrix} 
\delta_N  \\
 \delta_{N-1}  \\
 \vdots \\
 \delta_{-N}
\end{pmatrix}
=
\begin{pmatrix} 
0  \\
0 \\
 \vdots   \\
 0
\end{pmatrix}
\end{equation*}
As in the proof of Theorem \ref{singer}, we multiply by the matrix \eqref{chimatrix} and obtain the unique solution $\delta_i=0$ for all $i$.

It is clear that $\Phi $ is injective because $\Phi(\sum_i x_i\tilde{w}_{k-i})$ is of the form $\sum_i t^{k-i-1} \otimes x_i $ plus terms involving only powers of $t$ that are strictly smaller than the largest power occuring in this sum. Thus it is non-zero. Since $\Phi $ is a map between vector spaces that are finite and of the same dimension in each degree, it is also an isomorphism.

Since $\epsilon $ is an Ext-isomorphism, so is $(\widehat{f^*\varphi})^*$. 
\end{proof}

\begin{cor}
$\widehat{f^*\varphi} : f^*MT(d)_2^\wedge \to \widehat{f^*MT}(d)_2^\wedge $ is a weak $(d-1)$-equivalence.
\end{cor}

\begin{proof}
According to Theorem \ref{invadams}, there is a spectral sequence  with $E_2$-term 
\begin{equation*}
E^{s,t}_2=\Ext_\mathcal{A}^{s,t}(WH(X)(d)^*,\Z/2)
\end{equation*}
and converging strongly to $\pi_{t-s}(\widehat{f^*MT}(d)_2^\wedge)$. Similarly, there is a spectral sequence converging strongly to $\pi_{t-s}(f^*MT(d)_2^\wedge)$ with 
\begin{equation*}
\bar{E}^{s,t}_2=\Ext_\mathcal{A}^{s,t}( H^*(f^*MT(d)) ,\Z/2)
\end{equation*}
and a map $\widehat{f^*\varphi}_{*}:\bar{E}_2^{s,t}\to E_2^{s,t}$ between them.

From the long exact sequence of $\Ext$-groups associated to the short exact sequence $0\to K \to WH(X)^* \to WH(X)(d)^* \to 0$, it follows that
\begin{equation*}
\Ext_\mathcal{A}^{s,t}(WH(X)(d)^* ,\Z/2) \to \Ext_\mathcal{A}^{s,t}(WH(X)^*,\Z/2)
\end{equation*}
is an isomorphism for $t-s < d$. Similarly for 
\begin{align*}
\Ext_\mathcal{A}^{s,t}( H^*(f^*MT(d)) ,\Z/2) \to \Ext_\mathcal{A}^{s,t}( H^*(f^*MT) ,\Z/2) . 
\end{align*}

Combined with Theorem \ref{extiso}, this shows that also $\widehat{f^*\varphi}_{*}: \bar{E}_2^{s,t}\to E_2^{s,t}$ is an isomorphism for $t-s < d$. Thus it it a $(d-1)$-equivalence on $E_\infty^{s,t}$. Since both spectral sequences converge strongly, we conclude by \cite{boardman}, Theorem 2.6, that
\begin{equation}\label{limiso}
\widehat{f^*\varphi}_{*} :\pi_{t-s}(f^*MT(d)_2^\wedge) \to \pi_{t-s}(\widehat{f^*MT}(d)_2^\wedge)
\end{equation} 
is an isomorphism for $t-s < d-1$. Since all the $\Ext$-groups are finite, a modification of Boardman's proof also shows that $\widehat{f^*\varphi}_{*}$ is a surjection for $t-s=d-1$.
\end{proof}
%


\section{The non-compact case}\label{nonres}
Note how compactness of $X$ was essential in the proof of Theorem \ref{singer} in order to make $\Phi$ well-defined and show that it is an isomorphism. This section is devoted to the non-compact case, in which the spectra behave completely differently.
\subsection{The unoriented case}
We first consider the unoriented spectra $MTO$. Now $H^*$ will denote $H^*(MTO)$, while $WH^*$ is as in Definition~\ref{defWH}. 
$WH^*$ might still look like the Singer construction applied to $H^*$. However, we shall see that the algebraic behavior is very different. In fact, $WH^*$ has infinite $\Ext^{0,*}_\A$-groups in all dimensions. 

It is well-known that $H^*$ is a free $\A$-module, see e.g.\ \cite{stong}. The set $\Hom_\A^*(H^*,\Z/2)$ is known to be a polynomial algebra $\Z/2[\xi_k,\, k\neq 2^s-1]$. The multiplication comes from the map $\Delta^* : H^* \to H^* \otimes H^*$ induced by the direct sum map $BO\times BO \to BO$. This $\Delta^* $ is the $H^*(BO)$-linear map given by the formula \begin{equation}\label{BOmult}
w_k \mapsto \sum_j w_j \otimes w_{k-j}
\end{equation}
where $w_k\in H^*$ corresponds to the $k$th Stiefel--Whitney class in $H^*(BO)$ under the Thom isomorphism. 
The classical way to prove that $H^*$ is free is by an algebraic study of the map $\Delta^*$, see \cite{stong}, Chapter VI. In this section we try to generalize this to $WH^*$. However, some difficulties occur because $WH^*$ is not bounded below.

The map $MT \to \widehat{MT}$ induces a surjective $\mathcal{A}$-homomorphism $WH^* \to H^*$. Thus there is an injection $\Hom_\A^*(H^*,\Z/2)\to \Hom_\A^*(WH^*,\Z/2)$. We even have:

\begin{thm}
There is an $\A$-homomorphism ${\Delta} : WH^* \to H^* \hat{\otimes} WH^*$ where $\A$ acts on the right hand side by the Cartan formula. It is also a map of $H^*(BO)$-modules where $H^*(BO)$ acts on the right hand side by the formula \eqref{BOmult}. This makes $\Hom_\A^*(WH^*,\Z/2)$ into a module over $\Hom_\A^*(H^*,\Z/2)$. 
\end{thm}

Here $H^*\hat{\otimes}WH^*$ denotes the inverse limit $\varprojlim_d H(d)^*\otimes WH^*$. One can think of this as a submodule of $\prod_k H^k \otimes WH^{*-k}$.

\begin{proof}
The direct sum map $ BO(d') \times BO(d+ka_r) \to BO(d' + d + ka_r)$ induces a map 
\begin{equation*}
MT(d')\wedge MT(d + ka_r,r) \to MT(d' + d + ka_r,r).
\end{equation*}
This defines an $\A$-homomorphism 
\begin{equation*}
H^*(MT(d' + d + ka_r ,r))\to H^*(MT(d'))\otimes H^*(MT(d + ka_r,r))
\end{equation*}
given by the formula \eqref{BOmult}.
This map commutes with the periodicity maps and 
the maps in the inverse system, inducing an $\A$-homomorphism
\begin{equation*}
WH(d'+d)^* \to H(d')^* \otimes WH(d)^* .
\end{equation*}
Taking the inverse limit over $d$, we get a map $WH^* \to H(d')^*\otimes WH^*$, and the inverse limit over $d'$ yields the desired map $WH^*\to H^*\hat{\otimes} WH^*$. 

The $H^*(BO)$-module structure comes from the commutative diagram
\begin{equation*}
\xymatrix@C=14pt{{MT(d')\wedge MT(d + ka_r)}\ar[r]\ar[d]&{MT(d') \wedge MT(d + ka_r)\wedge( BO(d')\times BO(d + ka_r))_+}\ar[d]\\
{MT(d' + d + ka_r )}\ar[r]&{MT(d' + d + ka_r)\wedge BO(d' + d + ka_r)_+.}
}
\end{equation*}
\end{proof}

Next we construct an infinite collection of elements in $\Hom_\A^t(WH^*,\Z/2)=\Ext_\A^{0,t}(WH^*,\Z/2)$ in each dimension $t$. 
The formula
\begin{equation*}
\Sq^2(1)=\bar{w}_2 = w_2 + w_1^2
\end{equation*}
holds in $H^*$.
Thus there is an $\A$-homomorphism $\xi_2 : H^2 \to \Z/2$ taking the value 1 on both $w_2$ and $w_1^2$. This element plays a special role. 

\begin{thm}\label{xi2}
$\xi_2$ is invertible in $\Hom_\A^*(WH^*,\Z/2)$, i.e.\ there are elements $\xi_2^{-n}$ satisfying $\xi_2^{m} \cdot \xi_2^{-n} = \xi_2^{m-n}$ for all $n,m \in \N$. The monomials $\xi_2^{n} \xi_I$ for $n \in \Z$ and $\xi_I \in \Z/2[\xi_4,\xi_5,\dots ]$ are linearly independent in $\Hom_\A(WH^*,\Z/2)$. Moreover, the $\Hom_\A^*(H^*,\Z/2)$-module structure of $\Hom_\A^*(WH^*,\Z/2)$ extends to a module structure over $\Hom_\A^*(H^*,\Z/2)[\xi_2^{-1}]$.
\end{thm}

\begin{proof}
We must determine the value of $\xi_2^{-n}(p_i\tilde{w}_{-i-2n})$  for every $p_i \in H^i(BO)$. This comes from some $p_iw_{ka_r - i - 2n} \in H^{ka_r-2n}(MT(d + ka_r,r))$ in the direct system \eqref{navn} and here we have the $\A$-homomorphism
\begin{equation}\label{xi-n}
H^{ka_r-2n}(MT(d + ka_r,r))\to H^{ka_r-2n}(MT(d + ka_r)) \xrightarrow{\xi_2^{(k\frac{a_r}{2}-n)}} \Z/2.
\end{equation}
Define 
\begin{equation*}
\xi_2^{-n}(p_i\tilde{w}_{-i-2n})= \xi_2^{(k\frac{a_r}{2}-n)} (p_iw_{ka_r - i - 2n} ).
\end{equation*}
We must check that this definition does not depend on the choice of $ka_r$ for $r$ large enough.

The $(k\frac{a_r}{2}-n)$-fold direct sum map $ BO\times \dotsm \times BO \to BO$ induces
\begin{equation*}
\Delta^*: H^{ka_r-2n}(MT(d + ka_r))\cong H^{ka_r-2n}(MTO) \to H^2(MTO)^{\otimes (k\frac{a_r}{2}-n)}.
\end{equation*}
To evaluate $\xi_2^{(k\frac{a_r}{2}-n)} (p_iw_{ka_r - i - 2n} )$, we must apply $\xi_2$ to each factor of 
\begin{equation*}
\Delta^*(p_iw_{ka_r - i - 2n} )=\Delta^*(p_i)\Delta^*(w_{ka_r - i - 2n} ).
\end{equation*}
 This yields the formula
\begin{equation}\label{xiformel} 
\xi_2^{(k\frac{a_r}{2}-n)}(p_i\tilde{w}_{ka_r-i-2n}) = \sum_{a+2b=i}\binom {k\frac{a_r}{2} - n}{a,b}\xi_1^a\xi_2^b(p_i).
\end{equation}
Here 
\begin{equation*}
\xi_1^a\xi_2^b : H^{i}(BO) \xrightarrow{\Delta} H^0(BO)^{\otimes (k\frac{a_r}{2} - n-a-b)}\otimes H^1(BO)^{\otimes a} \otimes H^2(BO)^{\otimes b} \to \Z/2
\end{equation*}
is the evaluation in each factor of the $\Z/2$-linear maps $\xi_1: H^1(BO) \to \Z/2$ and $\xi_2 :H^2(BO) \to \Z/2$, given by $\xi_1(w_1)=1$, $\xi_2(w_1^2)=1$, and $\xi_2(w_2)=1$.

The multinomial coefficient in \eqref{xiformel} only depends on $(k\frac{a_r}{2} - n) \bmod 2^N$ where $N$ is the smallest number such that $i<2^N $. But $2^N$ divides $\frac{a_r}{2}$ for all $r$ sufficiently large. Also, $\xi_1^a\xi_2^b(p_i)$ is independent of $ka_r$ when this is larger than $i$, since a larger $ka_r$ only will add more copies of $1\in H^0(BO)$ to the formula for ${\Delta}^*(p_i)$. Hence \eqref{xiformel} does not depend on $ka_r$ for $r$ sufficiently large, and we have a well-defined $\Z/2$-linear map.

We must see that this map is actually an $\A$-homomorphism. Let $x \in WH^{-2n-j}$ be given. We must see that $\xi_2^{-n}(Sq^j(x))=0$. For this, choose $d,r$ and $k$ so large that $x$ comes from some $x'\in H^*(MT(d+ka_r,r))$ and so that we may compute
\begin{equation*}
\xi_2^{-n}(\Sq^j(x))=\xi_2^{(k\frac{a_r}{2}-n)}(\Sq^j(x')).
\end{equation*}
This is zero because $\xi_2^{(k\frac{a_r}{2}-n)}$ is an $\A$-homomorphism.

Now we look at how the $\xi_2^{-n}$ multiply. This is given by the map 
\begin{equation*}
WH^* \to H^* \hat{\otimes} WH^* \to H^{m} \otimes WH^{-2n} \xrightarrow{\xi_2^{m}\otimes \xi_2^{-n}}  \Z/2.
\end{equation*}
For sufficiently large $d',d,r$ and $N$, we have the following diagram
\begin{equation*}
\xymatrix{{H^*(MT(d'+d + 2^N ,r))}\ar[r]\ar[d]&{H^*(MT(d'))\otimes H^*(MT(d + 2^N,r))}\ar[d]\\
{H^*(MT(d' + d + 2^N ))}\ar[r]\ar[dr]_{\xi_2^{(m+2^{N-1}-n)}}&{H^*(MT(d'))\otimes H^*(MT(d + 2^N))}\ar[d]^{\xi_2^m \otimes \xi_2^{(2^{N-1} -n)}}\\
{}&{\Z/2.}
}
\end{equation*}
But $\xi_2^{(2^{N-1}+m-n)}(p_iw_{2^{N} + k}) = \xi_2^{(m-n)}(p_iw_{ k})$. This is by definition when $m-n$ is negative and because \eqref{xiformel} also holds for $m-n$ positive. Thus the result follows by commutativity of the diagram.

Finally, the linear independence follows from this multiplicative structure. Suppose a finite sum of monomials $\sum a_{n,I}\xi_2^{-n}\xi_I=0$ is given. We can multiply this by a large power of $\xi_2$ so that the sum only contains positive  powers of $\xi_2$. Since all such monomials are linearly independent, all $a_{n,I}$ must be 0. 

Let $\xi \in \Hom_\A(H^*,\Z/2)[\xi_2^{-1}]$ and $\eta \in \Hom_\A(WH^*,\Z/2)$. Then for $p_i \tilde{w}_l \in WH^*$ we define 
\begin{equation*}
\xi \cdot \eta(p_i \tilde{w}_l) = \xi_2^{2^{N-1}} \xi \cdot \eta (p_i \tilde{w}_{l+2^N})
\end{equation*}
for $N$ large. As before, one can write down the formulas to see that, for a given $p_i \tilde{w}_l$, it is independent of $N$ for $N$ large and it has the desired properties. 
\end{proof}

Let $\A(n)$ denote the Hopf subalgebra of $\A$ generated by the elements $\Sq^{2^i}$ for $i\leq n$. This is finite by \cite{milnor}, and $\A $ is free over each $\A(n)$ by general results on Hopf algebras given in \cite{moore}. 

We want to generalize the proof that $H^*$ is free over $\A$ to $WH^*$. However, since $WH^*$ is not bounded below, the proof only works over the finite subalgebras $\A(n)$. It turns out that this is enough to compute the $\Ext $-groups we are after.

\begin{lem}\label{free}
$WH^*$ is free over $\A(n)$. 
\end{lem}

\begin{proof}
Let $D^t_n = \A(n)WH^{<t} + WH^{>t}$. This is a submodule over $\A(n)$. Therefore, $M^t_n = WH^* / D^t_n$ is again an $A(n)$-module with trivial action, and the projection $WH^*\to M^t_n$ is $\A(n)$-linear. Since $\A$ is free over $\A(n)$, $H^*$ is also free over $\A(n)$, so we may choose an $\A(n)$-linear projection $H^* \to \A(n)$. Thus there is a well-defined $\A(n)$-homomorphism 
\begin{equation}\label{An}
WH^* \to H^* \hat{\otimes} WH^* \to \A(n){\otimes} M^t_n.
\end{equation} 
This is clearly surjective, since a lift $M^t_n \to WH^t$ of the projection defines an $\A(n)$-linear map $i_t :\A(n) \otimes  M^t_n  \to WH^*$ whose composition with \eqref{An} is clearly the identity. Then also 
\begin{equation}\label{sumM}
WH^* \to \bigoplus_{t=N}^\infty \A(n) \otimes M^t_n 
\end{equation}
must be surjective. This can be seen using the splitting $i_t$ since the $\A(n) \otimes M^{N}_n $ summand is hit by $i_{N} (\A(n) \otimes M^{N}_n)$. Then $i_{N+1}(\A(n)  \otimes M^{N+1}_n)$ hits the $\A(n) \otimes M^{N+1}_n $ summand mod $\A(n) \otimes M^{N}_n $. Continuing this way, we see that it is surjective in each dimension. 
 Letting $N \to -\infty$, finiteness of $\A(n)$ implies that
\begin{equation*}
WH^* \to \bigoplus_{t=-\infty}^\infty \A(n)\otimes  M^t_n
\end{equation*}
is a surjection.
Let $K$ denote the kernel. Then $WH^*$ splits as a sum  
\begin{equation*}
K \oplus \bigoplus_{t=\infty}^\infty i_t(\A(n) \otimes M^t_n )
\end{equation*}
 of $\A(n)$-modules. 
Every element in $K$ must be decomposable, otherwise they would be non-zero in some $M^t_n$. Because of the splitting, we must have $\A(n)^{>0}K=K$. Iterating this, we get $(\A(n)^{>0})^lK=K$ for all $l$. But $\A(n)$ is finite, so $K$ must be zero.
\end{proof}

Define $M^* =\varinjlim_n M_n^* = WH^*/\A^{>0}WH^*$.  In the following, the $\Z/2$-dual of a $\Z/2$ vector space $V$ will be denoted by $V^\vee$.
 
\begin{thm}
\begin{equation*}
\Ext_\A^{s,t} (WH^*,\Z/2)=\begin{cases}
(M^t)^\vee & \textrm{ for } s=0,\\
0&\textrm{ for } s>0.
\end{cases}
\end{equation*}
\end{thm}

\begin{proof}
It follows from Lemma \ref{free} that 
\begin{equation*}                   
\Tor_{s,t}^{\A(n)}(\Z/2^\vee,WH^*)=
\begin{cases}
WH^t/\A(n)^{>0}WH^* &\textrm{ for }s=0,\\
0&\textrm{ for } s>0.
\end{cases}
\end{equation*}
There is an isomorphism 
\begin{equation*}
\varinjlim_n \Tor_{s,t}^{\A(n)}(\Z/2^\vee,WH^*) \to \Tor_{s,t}^{\A}(\Z/2^\vee,WH^*),
\end{equation*}
see e.g.\ \cite{adams}.
This allows us to calculate 
\begin{equation*}
\Tor_{0,t}^{\A}(\Z/2^\vee,WH^*)\cong \varinjlim_n WH^t/\A(n)^{>0}WH^* \cong WH^t/\A^{>0} WH^* = M^t.
\end{equation*}
But again by \cite{adams}, Lemma 4.3,  
\begin{equation*}
\Ext^{s,t}_\A(WH^*,\Z/2) \cong \Hom(\Tor_{s,t}^{\A}(\Z/2^\vee,WH^*),\Z/2),
\end{equation*}
and the claim follows.
\end{proof}

\begin{cor}
For all $t < d$, $\pi_t(\widehat{MT}(d)_2^\wedge) \cong (M^t)^\vee$. 
\end{cor}

\begin{proof}
We have that $\Ext^{s,t}_\A(WH(d)^*,\Z/2) \cong \Ext^{s,t}_\A(WH^*,\Z/2)$ whenever $t-s < d$, so in these dimensions the spectral sequence of Theorem \ref{invadams} is concentrated on the line $s=0$. Hence there can be no non-trivial differentials in this area. 
\end{proof}

\subsection{More on the structure of $WH^*$}\label{dual}
We saw in Section \ref{nonres} that $WH^*$ has the $\Ext $-groups of a free $\A$-module. To see that it is in fact a free module and find an explicit description of $M^t$, we need a more convenient basis for $H^*$. 

Let $v_{i}\in H_{i}(BO)$ be the image of the generator in $H_{i}(BO(1))\cong \Z/2$. The direct sum map $BO \times BO \to BO$ induces a multiplication in $H_*(BO)$ dual to the comultiplication $\Delta^*$ in $H^*(BO)$, and the dual of the cup product in $H^*(BO)$ defines a comultiplication $\Delta_*$ on $H_*(BO)$.

\begin{prop} 
$H_*(BO)$ is a polynomial algebra on the generators
$v_{i}$. The comultiplication is given by 
$\Delta_*(v_{i})=\sum_{j}v_{j}\otimes v_{i-j}$.
\end{prop}
See e.g.\ \cite{switzer}, Chapter 16.
We will need the following relations to the Stiefel-Whitney classes:
\begin{lem}
\label{le:formulas}
\InsertTheoremBreak
\begin{enumerate}
\item [(i)]
\[
w_{i}(v_{i})=
\begin{cases}
1 &\text{ if $i=1$,}\\
0 &\text{ otherwise.}
\end{cases}
\]
\item[(ii)] 
If $j\geq 2$ and $x\in H_{*}(BO)$, then $w_{n}(v_{j}x)=0$.
\item[(iii)] If $x\in H_{*}(BO)$, $a=\prod_{1\leq k\leq n}w_{i_{k}}$, and
$j \geq n$, then $a(v_{j}x)=0$.
\item[(iv)] Assume that $2^{N}>i_{j}$ for all $j>1$. Then
\[
w_{i_{1}+2^{N}}w_{i_{2}}\dotsm w_{i_{n}}(v_{1}^{2^{N}}x)=
w_{i_{1}}w_{i_{2}}\dotsm w_{i_{n}}(x)
\]
and
\[
w_{i_{1}+k2^{N}}w_{i_{2}}\dots w_{i_{n}}(v_{k}^{2^{N}}x)=0.
\]
\end{enumerate}
\end{lem}

\begin{proof}
The first formula follows from the definition of $v_{i}$ because the
restriction of $w_{i}$ to $H^i(BO(1))$ is non-trivial if and only if $i=1$.

Formula (ii) follows from (i) by using the comultiplication in $H^{*}(BO)$:
\begin{equation*}
w_{n}(v_{j}x)=\Delta^*(w_n)(v_jx)=w_{j}(v_{j})\cdot w_{n-j}(x)=0. 
\end{equation*}

Formula (iii)
 follows from (ii) since
\[
a(v_{j}x)=(\prod_{1\leq k\leq n}w_{i_{k}})(\Delta_* v_{j}\Delta_* x)
\]
is a sum of terms of the form 
$\prod_{1\leq k\leq n}(w_{i_{k}}(v_{i_{k}}x_k))$ with
$\sum_{1\leq k \leq n} i_{k}=j$. But $j>n$ implies $i_{k}\geq 2$ for some $k$, so these terms all equal 0.

Finally, (iv) follows from (ii) and the formula
$\Delta_*(v_{k}^{2^{N}}x)= \Delta_*(v_k)^{2^N}\Delta_*(x)$ where only the term $(v_{k}^{2^{N}}\otimes 1) \Delta_*(x)$ contributes when we evaluate $w_{i_{1}+k2^{N}}w_{i_{2}}\dotsm w_{i_{n}}$.
\end{proof}

We may filter $H^*$ by length of monomials
\[
H^{*}[n]=\Z/2 \{w_{i_{1}}\dotsm w_{i_{n}}\in H^{*}\}.
\]

\begin{lem}
Evaluation induces a perfect pairing in each dimension
\[
\mu_{n}:\Z/2[v_{1},v_{2},\dots,v_{n}]\otimes H^{*}[n]\to \Z/2.
\]
\end{lem}

\begin{proof}
Evaluation defines a perfect pairing
\begin{equation*}
\mu:\Z/2[v_{1},v_{2},\dots]\otimes H^{*}\to \Z/2.
\end{equation*}
The restriction to $H^{*}[n]\subset H^{*}$ yields a map 
\[
\mu_{n}:\Z/2[v_{1},v_{2},\dots,v_{n}]\otimes H^{*}[n]\to \Z/2.
\]
According to in Lemma \ref{le:formulas} (iii), the adjoint map
\[
\mu_{n}^{*}:H^{*}[n] \to
\Hom(\Z/2[v_{1},v_{2},\dots,v_{n}],\Z/2)
\]
must be injective, since the isomorphism $\mu^{*}$ is.
In each degree, $\Z/2[v_{1},v_{2},\dots,v_{n}]$ and
$H^{*}[n]$ are finite of the same dimension, 
so $\mu_{n}^{*}$ is an isomorphism, i.e.\ $\mu_n$ is a perfect
pairing.
\end{proof}

We now want to a give a similar description of $WH^*$ and its dual. 
By Theorem~\ref{coholim}, $WH^*$ has a $\Z/2$-basis 
\[
\{\tilde{w}_{i_{1}}w_{i_{2}}\dotsm w_{i_n} \mid i_{1}\in \Z, n \in \N, \ i_{l}\geq 0
\text{ for } l\geq 2\}.
\]
Again there is a filtration of $WH^*$ by length of monomials
\[
WH^{*}[n]=\Z/2 \{\tilde{w}_{i_{1}}\dotsm w_{i_{n}}\mid i_{1}\in \Z,  \ i_{l}\geq 0
\text{ for } l\geq 2\}.
\]
$WH^{*}[n]$ is not an $\A$-module, but $\Sq^k : WH^{*}[n] \to WH^{*}[n+k]$ by Formula \eqref{Awl} and \eqref{binom}.
%

Let $\xi_2\in H_2(BO)$ be as in the previous section. Then $\xi_2=v_{1}^{2}+v_{2}$. 
\begin{defi}
\begin{equation*}
\begin{split}
R(n)&=\Z/2[v_{1},v_{2},\dots,v_{n}][\xi_2^{-1}]\\
R&=\varinjlim_{n} R(n).
\end{split}
\end{equation*}
\end{defi}

The Steenrod algebra acts on $R(n)$, and the inclusion maps $R(n)\to R(n+1)$ are compatible with this action.
Thus $R$ has a natural action of the Steenrod algebra. 

The pairing $\mu_{n}$ extends to a pairing
\[
\tilde{\mu}_{n}:R(n)\otimes WH^{*}[n]\to \Z/2
\]
by the formula 
\[
\tilde{\mu}_n(\xi_2^{k}x,\tilde{w}_{i_{1}}w_{i_{2}}\dotsm w_{i_{n}})
=\mu_n(\xi_2^{2^{N-1}+k}x,{w}_{i_{1}+2^{N}}w_{i_{2}}\dotsm w_{i_{n}})
\]
where $x\in H_*(MTO)$ and $N$ is sufficiently big. 
Repeated use of Lemma \ref{le:formulas} (iv)
shows that this definition is independent of $N$.
The pairing induces a natural map $\tilde{\mu}^{*}:WH^{*}\to \Hom(R,\Z/2)$ compatible with the Steenrod action.

\begin{lem}
\label{le:topology}
For $a\in WH^{*}$, there are natural numbers $m,n$ such that
\begin{equation*}
\tilde{\mu}^{*}(a) \in \Hom(R(m,n),\Z/2) \subseteq \Hom(R,\Z/2)
\end{equation*}
where $R(m,n) = R/I(m,n)$ and $I(m,n)$ is the ideal in $R$ generated by 
\begin{equation*}
 v_{2}^{m},v_{3}^{m},\dots,v_{n}^{m},v_{n+1},v_{n+2},\dots 
\end{equation*}
\end{lem}

\begin{proof}
Choose $n$ such that $a\in WH^{*}[n]$. Then by Lemma \ref{le:formulas} (iii),
$\tilde{\mu}^{*}(a)(v_{j}x)=0$ for $j>n$. 

Assume without restriction that
$a=\tilde{w}_{i_{1}}w_{i_{2}}\cdots w_{i_{n}}$. 
Choose $m=2^{N} > i_{j}$ for all $j>1$. Then
$\tilde{\mu}^*(a)(v_{i}^{m}x)=0$ by Lemma \ref{le:formulas} (iv). 
\end{proof}

\begin{defi}
Let $R^\wedge = \varprojlim_{m,n} R(m,n)$ with the inverse limit topology. 
If $x \in I(m,n)$, then $\Sq^k(x) \in I(m-k,n-k)$. This defines a continuous action of the Steenrod algebra on $R^\wedge$. Let 
\begin{equation*}
\Hom^{\topo}(R,\Z/2)= \varinjlim_{m,n} \Hom (R(m,n),\Z/2) \subseteq \Hom(R,\Z/2)
\end{equation*}
 denote the $\A$-module of continuous homomorphisms.
\end{defi} 
%
Lemma \ref{le:topology} implies that $\tilde{\mu}$ takes values in $\Hom^{\topo}(R,\Z/2)$. The following theorem is the key ingredient in the description of $WH^*$:

\begin{thm}\label{tildeiso}
The map $\tilde{\mu}^* :WH^*\to \Hom^{\topo}(R,\Z/2)$ is an isomorphism.
\end{thm}
%

We first give two equivalent filtrations of $R^\wedge$. The first one is convenient for describing the $\A$-action on $\Hom^{\topo}(R,\Z/2)$, while the second is useful in the proof of Theorem \ref{tildeiso}. First of all, we  need to describe the $\A$-module $H_*=H_*(MTO)$.

\begin{lem}
The Steenrod algebra acts on $H_*(MTO)$ by
\[
\Sq^{k}_{*}v_{i}=\binom{i-k-1}{k}v_{i-k}.
\]
\end{lem}
\begin{proof}
In $H^*(BO)$ the formula $\Sq^k(w_1^{i-k})= \binom{i-k}{k}w_1^{i}$ holds. In $H^*$ the formula becomes
\[
\Sq^k(w_1^{i-k})= \sum_{l=0}^k \binom{i-k}{l}w_1^{i-k+l}\bar{w}_{l-k}.
\]
This restricts to 
\[
\Sq^k(w_1^{i-k})= \sum_{l=0}^k \binom{i-k}{l}w_1^{i} = \binom{i-k-1}{k} w_1^i
\]
in $H^*(MTO(1))$ because $\bar{w}_{l-k}$ restricts to $w_1^{l-k}$. Dualizing yields the formula.
\end{proof}
 
Let $S = \Z/2[v_{2},v_{3},\dots ]$ be the subalgebra in $ H_*$. 

\begin{cor}\label{subS}
$S$ is an $\A$-submodule of $H_*$.
As algebras with $\A$-action, 
\[
\Z/2[v_{1},v_{2},\dots,v_{n}]\cong
\Z/2[v_{1}]\otimes \Z/2[v_{2},v_{3},\dots,v_{n}].
\]
\end{cor}

\begin{proof}
Clearly, $\Z/2[v_1]$ is a subalgebra of 
$\Z/2[v_{1},v_{2},\dots,v_{n}]$ which is closed under the
action of $\A$. It is not quite as obvious that
the subalgebra $\Z/2[v_{2},v_{3},\dots ]$ is closed, but it follows
from the preceeding formula. If 
$\Sq^k_*(v_i)=v_j$ we have that $\binom{j-1}k=1$. So if $j=1$,
we must have $k=0$.
The inclusions combine to an isomorphism of algebras
\[
\Z/2[v_{1}]\otimes \Z/2[v_{2},v_{3},\dots,v_{n}]
\to \Z/2[v_{1},v_{2},\dots,v_{n}]
\]
 compatible with the $\A$-action.
\end{proof}

It follows from Corollary \ref{subS} that $S\{1,v_1\}=S\oplus Sv_1$ is an $\A$-submodule of $H_*$.

\begin{prop}\label{Rhat}
$R^\wedge$ is isomorphic as a topological $\A$-module to the completion of $R$ with respect to the subspaces
\begin{equation*}
J_r = \{\xi_2^m p_0 + \dots + \xi_2^{m-k}p_k \mid p_i \in S\{1,v_1\},\ \deg(p_i)\geq r \}.
\end{equation*}
One can think of $R^\wedge$ as the space of power series $\xi_2^m p_0 + \xi_2^{m-1}p_1+ \dotsm$  with coefficients $p_i \in S\{1,v_1\}$ and $\A$ acting on the coefficients. 
\end{prop}
Note, however, that the ring structure is not clear from this description, as $S\{1,v_1\}$ is not closed under multiplication.
\begin{proof}
The topologies agree because 
\begin{align*}
J_r\subset I(m,n) \quad&\text{if $r > \frac{n(n+1)}{2}m$},\\
I(m,n)\subset J_r \quad&\text{if $\min\{m,n\} > r$}.
\end{align*}
Furthermore, if $x \in J_r$, then $\Sq^k_*(x) \in J_{r-k}$. Thus we get the same $\A$-action.

Observe that
\begin{equation*}
R \cong \Z/2[\xi_2,\xi_2^{-1}] \otimes S\{1,v_1\}
\end{equation*}
as $\A$-modules. Thus any element of $R$ has a unique representation of the form $\xi_2^m p_0 + \dots + \xi_2^{m-k}p_k$.
It is then clear that the completion of $R$ in the $J_r$'s is the space of power series of the form $\xi_2^m p_0 + \xi_2^{m-1}p_1+ \dotsm$ with $p_i \in S\{1,v_1\}$ and $\A$ acting on coefficients, since it acts trivially on $\xi_2$. 
%
\end{proof}

\begin{prop}\label{R1}
$R^\wedge $ is isomorphic as a topological ring to the ring $R_1^\wedge$ of power series 
\[
v_1^mp_0+v_1^{m-1}p_1+\dots +v_1^{m-i}p_i+\dotsm
\] 
with coefficients $p_i \in S$.
\end{prop}

\begin{proof}
$R_1^\wedge$ is its own completion in the ideals 
\begin{equation*}
L_r = \{v_1^mp_0+v_1^{m-1}p_1+\dotsm  \mid p_i \in S,\ \deg(p_i)\geq r\},
\end{equation*}
while $R^\wedge $ is its own completion in the  
\begin{equation*}
J_r = \{\xi_2^mp_0+\xi_2^{m-1}p_1+\dotsm  \mid p_i \in S\{1,v_1\},\ \deg(p_i)\geq r\}.
\end{equation*}
But $v_1$ is invertible in $R^\wedge$ by the formula
\begin{equation*}
v_1^{-1} = v_1\xi_2^{-1}(1+v_2\xi_2^{-1} + v_2^{2}\xi_2^{-2}+\dotsm )
\end{equation*}
and $\xi_2$ is invertible in $R^\wedge_1$ via the formula
\begin{equation*}
\xi_2^{-1} = v_1^{-2}+ v_1^{-4}v_2 + v_1^{-6}v_2^2+ \dotsm.
\end{equation*}
Thus there are isomorphisms $J_r \to L_{r-1}$. This induces a continuous isomorphism $R_1^\wedge \cong R^\wedge$.
%
\end{proof}


\begin{proof}[Proof of Theorem \ref{tildeiso}]
We consider $R^\wedge$ with the topology given in Proposition \ref{R1}. 

Let ${\cal I}_n$ be the set of sequences $i_1,i_2,\dots,i_n$ 
with $i_1\in \Z$ and $i_k\geq 0$ for $k \geq 2$ such that 
$i_2\geq i_3\geq \dots \geq i_n$. For $I\in {\cal I}_n$ define
$w_I=\tilde{w}_{i_1}w_{i_2}\dotsm w_{i_n}$. Then the $w_I$ for $I\in {\cal I}_n$ 
form a basis for $WH^*[n]$. Let 
\begin{equation*}
l_s(I)=\sum_{k=s}^n i_k.
\end{equation*}
 Then the degree
of $w_I$ is $l_1(I)$. We can define a
partial order on ${\cal I}_n$ by saying that
$I\leq I^\prime$ if $l_1(I)=l_1(I^\prime)$ and $l_s(I)\leq l_s(I^\prime)$ for $2\leq s \leq n$.

Let ${\cal J}_n$ be the set of sequences $j_1,j_2,\dots,j_n$ where
$j_1\in \Z$ and $j_k\geq 0$ for $k\geq 2$. Let 
$v^J=v_1^{j_1}v_2^{j_2}\cdots v_n^{j_n}$. The set of $v^J$ for $J\in {\cal J}_n$
constitutes a basis for $\Z/2[v_1,v_2,\dots,v_n][{v_1}^{-1}]$. Let
\begin{equation*}
l_s(J)=\sum_{k=s}^n (k-s+1) j_k. 
\end{equation*}
The degree of $v^J$ is $l_1(J)$. 

There is a bijection $\alpha:{\cal J}_n\to {\cal I}_n$ given by
$\alpha(j_1,j_2,\dots,j_n)=(i_{1},i_{2},\dots,i_{n})$ with
$i_{k}=\sum_{k\leq m \leq n} j_{m}$.
Note that $l_s(\alpha(J))=l_s(J)$. Give ${\cal J}_n$ 
the partial order
$J\leq J^\prime$ if $l_s(J)\leq l_s(J^\prime)$ for $2\leq s \leq n$ and $l_1(J)=l_1(J^\prime)$.
Then  $\alpha$ preserves the partial order and so does the inclusion ${\cal J}_n \to {\cal J}_{n+1}$. In particular, $\alpha$ induces a bijection $\varinjlim_n {\cal J}_n \to \varinjlim_n {\cal I}_n $.

We claim that
\begin{equation}
\label{wIvJ}
w_I(v^J)=
\begin{cases}
1 &\text{ if $\alpha(J)=I$,}\\
0 &\text{ unless $\alpha(J) \leq I$}.
\end{cases}
\end{equation}
To prove this formula, we first compute using Lemma \ref{le:formulas}:
\[
\tilde{w}_{i_1}w_{i_2}\cdots w_{i_n}(v_1^{j_1}v_2^{j_2}\cdots v_n^{j_n})=
\begin{cases}
\tilde{w}_{i_1-j_n}w_{i_2-j_n}\cdots w_{i_n-j_n}(v_1^{j_1}v_2^{j_2}\cdots v_{n-1}^{j_{n-1}})
&\text{$j_n\leq i_n$,}\\
0&\text{$j_n > i_n$.}
\end{cases}
\]
The claim \eqref{wIvJ} follows for $\alpha(J)=I$ by induction on $n$.

Now assume $w_I(v^J)\neq 0$. We claim that for all $1\leq s\leq n$,
\begin{equation}\label{induktion}
\tilde{w}_{i_1}w_{i_2}\cdots w_{i_n}(v_1^{j_1}v_2^{j_2}\cdots v_n^{j_n}) = \sum \tilde{w}_{i_1^\prime}w_{i_2^\prime}\cdots w_{i_n^\prime}(v_1^{j_1}v_2^{j_2}\cdots v_s^{j_s})
\end{equation}
where the sum is over some non-empty set of $(i_1',\dots,i_n')\in {\cal J}_n$  satisfying for all $2 \leq t\leq s+1$
\begin{equation}\label{l(I)}
\sum_{k\geq t}i_k^\prime \leq \sum_{k\geq t} i_k - \sum_{k \geq s+1} (k-t+1)j_k. 
\end{equation}
This is clearly true for $s=n$. Assume that this is true for some $s$.
To compute \eqref{induktion}, we must evaluate on $\Delta_*(x)\Delta_*(v_s)^{j_s}$ where $x = v_1^{j_1}v_2^{j_2}\cdots v_{s-1}^{j_{s-1}}$. In $\Delta_*(v_s)^{j_s}$, only terms of the form $\sum v_1^{l_1} \otimes \dotsm \otimes v_1^{l_n}$ contribute. The $l_k$ must satisfy $l_k \leq j_s$ and $\sum l_k = sj_s$. In particular,
\begin{equation}\label{suml_k}
\sum_{k \geq t} l_k \geq (s-t+1)j_s.
\end{equation}
Thus by Lemma \ref{le:formulas},
\begin{align*}
\tilde{w}_{i_1'}w_{i_2'}\cdots w_{i_n'}(v_1^{j_1}v_2^{j_2}\cdots v_s^{j_s}) =& \sum_l \tilde{w}_{i_1'}w_{i_2'}\cdots w_{i_n'}(\Delta_*(x) v_1^{l_1} \otimes \dotsm \otimes v_1^{l_n})\\=&\sum_l \tilde{w}_{i_1'-l_1}w_{i_2'-l_2}\cdots w_{i_n'-l_n}(x)
\end{align*}
where only terms with $i_k'-l_k \geq 0$ for all $k\geq 2$ contribute to the sum.
By \eqref{suml_k} and the induction hypothesis,
\[
\sum_{k\geq t}(i_k^\prime -l_k) \leq \sum_{k\geq t} i_k' - (s-t+1)j_s \leq \sum_{k\geq t} i_k - \sum_{k \geq s} (k-t+1)j_k. 
\]
Thus the claim follows for $s-1$.

For each $s\geq 1$, put $t=s+1$ in \eqref{l(I)}. This yields
\[
0 \leq \sum_{k\geq t} i_k^\prime  \leq l_{s+1}(I) - l_{s+1}(J) = l_{s+1}(I) - l_{s+1}(\alpha(J)) 
\]
where the first inequality follows because at least one term in the sum \eqref{induktion} is non-zero by the assumption $w_I(v^J)\neq 0$. But this means that $\alpha(J) \leq I$, proving \eqref{wIvJ}.

Let ${\cal J}_{n, J}=\{J^\prime\in {\cal J}_n\mid  J^\prime < J\}$ and note that this is a finite set.
Let 
\begin{equation*}
R_{n,J} = \Z/2 [v_1,v_2,\dots ][v_1^{-1}]/\Z/2 \{v^{J'} \mid J' \notin {\cal J}_{n, J}\}.
\end{equation*}
By Proposition \ref{R1}, there is an isomorphism of topological rings 
\begin{equation*}
R^\wedge \cong \varprojlim_{n,J} R_{n,J}.
\end{equation*}
For each $J\in {\cal J}_n$, define the homomorphism
$v^{J*}:\Z/2[v_1,v_2,\dots][{v_1}^{-1}] \to \Z/2$  by $v^{J*}(v^{J^\prime})=1$ if and only if $J=J^\prime$. Then the $v^{J'*}$ for $J'\in {\cal J}_{n,J}$ form a basis for the vector space $\Hom(R_{n,J},\Z/2)$. Thus the collection of all the $v^{J*}$ constitutes a basis for
\begin{equation*}
\Hom^{\textrm{top}}(R^\wedge,\Z/2) = \varinjlim_{n,J} \Hom(R_{n,J},\Z/2).
\end{equation*}
 
 We claim that
$\tilde{\mu}^*(w_{\alpha(J')})$ for $J'\in {\cal J}_{n,J}$ also form a basis for $\Hom(R_{n,J},\Z/2)$. 
 According to \eqref{wIvJ}, $\tilde{\mu}^*$ restricts to a map
\[\tilde{\mu}^*: \Z/2\{w_{\alpha(J^\prime)} \mid J^\prime\in {\cal J}_{n, J}\}\to
\Z/2\{v^{J^\prime*} \mid J^\prime\in {\cal J}_{n,  J}\}.
\]
Using \eqref{wIvJ}, an induction on $J$ shows that $\tilde{\mu}^*$ is a surjection and hence an isomorphism. Taking the direct limit over $n$ and $J$ finishes the proof of the lemma.
%
%
\end{proof}


Since $H^*$ is free over $\A$, $H_* \cong \Z/2[\xi_2,\xi_4,\xi_5,\dots]\otimes \A_*$ where $\A_*$ is the $\Z/2$-dual of $\A$.

\begin{lem}\label{Mv_1}
There is an isomorphism of $\A$-modules
\begin{equation*}
\phi : \Z/2[\xi_4,\xi_5,\dots]\otimes \A_* \cong S\{1,v_1\}.
\end{equation*}
\end{lem}
Note that it is not a ring isomorphism, as the right hand side is not closed under multiplication. 

\begin{proof}
Since
\begin{equation*}
\Z/2[\xi_2]\otimes S\{1,v_1\} \cong H_* \cong \Z/2[\xi_2] \otimes \Z/2[\xi_4,\xi_5,\dots] \otimes \A_*
\end{equation*}
as $\A$-modules,  dividing out by the $\A$-submodule $\xi_2\cdot H_*$ yields
\begin{equation*}
S\{1,v_1\} \cong H_*/\xi_2 \cdot H_* \cong \Z/2[\xi_4,\xi_5,\dots] \otimes \A_*
\end{equation*}
as $\A$-modules. 
\end{proof}

Let $S_\A = \phi(\Z/2[\xi_4,\xi_5,\dots])$.  We are now ready to give a full description of the $\A$-module structure of $WH^*$.

\begin{thm}
$WH^*$ is a free $\A$-module with
\begin{equation*}
WH^*/\A^{>0} WH^* \cong \Hom^{\topo}( S_{\A}\otimes \Z/2[\xi_2,\xi_2^{-1}],\Z/2)  
\end{equation*}
where $ S_{\A}\otimes\Z/2[\xi_2,\xi_2^{-1}]$ is topologized as the space of power series in $\xi_2$ with coefficients in $S_{\A}$.
\end{thm}

\begin{proof}
By Theorem \ref{Rhat}, $WH^* \cong \Hom^{\topo}(R^\wedge,\Z/2)$ where $R^\wedge$ can be thought of as the space of power series in $\xi_2^{\pm 1}$ with coefficients in $S\{1,v_1\}$. In each degree, this is the same as the completion of $R^\wedge$ in the subspaces
\begin{equation*}
I_l = \big\{\xi_2^{l-1}p_1 + \xi_2^{l-2}p_2 + \dotsm \mid p_k \in S\{1,v_1\}\big\}.
\end{equation*}
Note that $I_l$ is an $\A$-submodule of $R^\wedge$.
The composition
\begin{equation*}
\xi_2^l \cdot H_* \to R^\wedge \to R^\wedge/I_l 
\end{equation*}
is an isomorphism of $\A$-modules where
\begin{equation*}
\xi_2^l \cdot H_* \cong \Z/2\{\xi_2^k, k\geq l\} \otimes S_\A \otimes \A_*.
\end{equation*}

For each $l$, there is a split short exact sequence
\begin{equation*}
0\to \xi_2^{l+1}\cdot S\{1,v_1\} \to \xi_2^{l+1} \cdot H_* \to \xi_2^{l} \cdot H_* \to 0
\end{equation*}
of $\A$-modules. Thus 
\begin{equation*}
\xi_2^{l} \cdot H_* \cong \bigoplus_{k=l}^\infty \xi_2^k \cdot S\{1,v_1\}.
\end{equation*}
Furthermore, the diagram
\begin{equation*}
\xymatrix{{\xi_2^{l+1} \cdot H_*}\ar[r]\ar[d]&{R^\wedge/I_{l+1} }\ar[d]\\
{\xi_2^l \cdot H_*}\ar[r]&{R^\wedge/I_l }
}
\end{equation*}
commutes. This means that as $\A$-modules
\begin{equation*}
\begin{split}
\Hom^{\topo}(R^\wedge,\Z/2) \cong & \varinjlim_l (R^\wedge/I_l)^\vee \\
\cong &\varinjlim_l (\xi_2^l \cdot H_*)^\vee \\
\cong &\bigoplus_{k=-\infty}^\infty (\xi_2^k \cdot S\{1,v_1\})^\vee \\
\cong &\bigoplus_{k=-\infty}^\infty (\xi_2^k \cdot S_\A)^\vee \otimes \A,
\end{split}
\end{equation*}
proving the claim.
\end{proof}

\begin{cor}
\begin{equation*}
\Ext_\A^{0,*}(WH^*,\Z/2)  \cong \Hom_\A^*(WH^*,\Z/2) \cong \prod_{l=-\infty}^\infty \xi_2^l \cdot S_\A.
\end{equation*}
\end{cor}

\subsection{The oriented case}\label{MTor}
We now turn to the oriented situation. Hence, $WH^*$ will now denote the oriented version of the cohomology group from Definition \ref{defWH} and $H^*$ will denote $H^*(MTSO)$. In the unoriented case, we have seen that $WH^*$ is a free module, generalizing the situation for $H^*$. In the oriented case, the generalization from $H^*$ is not so obvious, but we do get something similar. The results we get in the oriented case are not as complete as those obtained in the unoriented case.

Recall that the $\A$-module $H^*$ is a direct sum of copies of $\A$ and $\A/\A\Sq^1$, see e.g.\ \cite{switzer}. In particular, the Thom class generates an $\A/ \A \Sq^1$ summand. This means that one can choose an $\A$-linear projection $H^* \to \A/ \A\Sq^1$ onto this summand.

There is map $WH^* \to H^* \hat{\otimes} WH^*$ constructed exactly as in the unoriented case. Combining this with the projection $H^* \to \A/ \A\Sq^1$, we get:
\begin{prop} \label{comod}
There is an $\A$-homomorphism
\begin{equation*}
WH^* \to  \A/\A\Sq^1 \hat{\otimes} WH^*.
\end{equation*}
\end{prop}

The sequence $H^* \xrightarrow{\Sq^1} H^* \xrightarrow{\Sq^1} H^*$ is not exact. In fact, the cohomology is a polynomial algebra on one generator in each dimension divisible by 4. This cohomology corresponds to the $\A/\A\Sq^1$ summands of $H^*$. 

In $WH^*$ we do not have a similar phenomenon:

\begin{prop}\label{nosplit}
The sequence $WH^* \xrightarrow{\Sq^1} WH^* \xrightarrow{\Sq^1} WH^*$ is exact.
\end{prop}

\begin{proof}
Let $x = \sum_i p_{2i}\tilde{w}_{2i} + \sum_i p_{2i+1}\tilde{w}_{2i+1}$ be some element of $WH^*$. If $x \in \Ker(\Sq^1)$, then
\begin{equation*}
0=\Sq^1(x)= \sum_i (\Sq^1(p_{2i}) \tilde{w}_{2i} + p_{2i}\tilde{w}_{2i+1})  + \sum_i \Sq^1(p_{2i+1})\tilde{w}_{2i+1}.
\end{equation*}
This happens precisely if $p_{2i} = \Sq^1(p_{2i+1})$ for all $i$. But then 
\begin{equation*}
x = \sum_i \Sq^1(p_{2i+1})\tilde{w}_{2i} + p_{2i+1}\tilde{w}_{2i+1} = \Sq^1(\sum_i p_{2i+1}\tilde{w}_{2i}),
\end{equation*}
proving the claim.
\end{proof}

However, there are relations. For instance, $\tilde{w}_0$ and $\tilde{w}_4$ are both indecomposable, but
\begin{equation*}
\Sq^{2,1,2}(\tilde{w}_0) = \Sq^1(\tilde{w}_4). 
\end{equation*}
In particular, there is no chance that $WH^*$ splits as a sum of a free module and some $\A/\A\Sq^1$ summands, but we shall see that the $\Ext $-groups behave as if it did. As in the unoriented case, the proof goes by considering $WH^*$ as a module over $\A(n)$.

\begin{lem}\label{sq1}
Suppose that $x_i \in WH^k$ represent linearly independent elements of $WH^k/\A(n) WH^{<k}$ with $\Sq^1(x_i)$ linearly independent in $WH^{k+1}/ \A(n) WH^{<k}$. Then $\sum_i a_i(x_i)\neq 0$ in $WH^*/ \A(n) WH^{<k}$ for any $a_i \in \A(n)$.
\end{lem}

\begin{proof}
The map in Proposition \ref{comod} takes $\A(n)WH^{<k}$  to $\A/\A \Sq^1 \hat{\otimes} \A(n)WH^{<k}$. Thus there is an $\A(n)$-linear map
\begin{equation*}
WH^*/\A(n)WH^{<k} \to \A/\A \Sq^1 {\otimes} WH^*/(\A(n)WH^{<k} + WH^{>k+1}).
\end{equation*}
Here $x_i$ maps to $1 \otimes x_i $ and thus $\Sq^1(x_i)$ maps to $1 \otimes \Sq^1(x_i)$, which is non-zero. 

If all $a_i \in \A(n) \Sq^1$, then $\sum_i a_i(x_i)$ maps to $\sum_i a_i'\otimes \Sq^1(x_i)$ where $a_i=a_i'\Sq^1$ and $a_i' \in \A(n)/\A(n) \Sq^1$ is non-zero (otherwise $a_i=0$). It follows from the assumptions that $\sum_i a_i'\otimes \Sq^1(x_i)$ is non-zero.

If at least one $a_j \notin \A(n) \Sq^1$, then $\sum_i a_i(x_i)$ maps to $\sum_i ( a_i \otimes x_i  + a_i'\otimes \Sq^1(x_i))$ for suitable $a_i'$. This is non-zero because at least one $a_j$ is non-zero in $\A(n)/ \A(n) \Sq^1$ and the $x_i$'s are linearly independent.
\end{proof}

Let $M_n^* = WH^*/\A(n)^{>0} WH^*$ be the space of indecomposables.
Define the subspace 
\begin{equation*}
K^k_n = \Ker(\Sq^1: M_n^k \to WH^{k+1}/ \A(n) WH^{<k} ).
\end{equation*}
Let $D_n^*$ be any complement so that $M_n^* = D_n^* \oplus K_n^*$ and choose a lift $M_n^* \to WH^*$ of the natural projection. This defines a map 
\begin{equation*}
\A(n) \otimes M_n^* \to WH^*.
\end{equation*}
This is surjective because $\A(n)$ is finite.

\begin{lem}\label{modD}
\begin{equation*}
\Tor_{s,t}^{\A(n)} (\Z/2^\vee,WH^*/\A(n) D_n^*) \cong K_n^{t-s}. 
\end{equation*}
The projection $WH^*/\A(n) D_n^* \to K_n^{t-s}$ induces an isomorphism
\begin{equation*}
\Tor_{s,t}^{\A(n)} (\Z/2^\vee,WH^*/\A(n) D_n^*) \to \Tor_{s,t}^{\A(n)}(\Z/2^\vee,K_n^{t-s}).
\end{equation*}
\end{lem}

\begin{proof}
We claim that the composition 
\begin{align*}
\A(n) \otimes K_n^l  \to WH^* \to WH^*/\A(n)(D_n^* \oplus K_n^{<l})
\end{align*}
 induces an injective $\A(n)$-homomorphism
\begin{equation*}
\A(n)/\A(n) \Sq^1 \otimes K_n^l \to WH^*/\A(n)(D_n^* \oplus K_n^{<l}).
\end{equation*}

Clearly, $\A(n)\Sq^1 \otimes K_n^l$ is in the kernel of the map by construction of $K_n^l$. On the other hand, assume $\sum a_i\otimes k_i$ maps to zero, i.e.\ there is a relation
\begin{equation}\label{summer}
\sum_i a_i(k_i) = \sum_j b_j(k_j) + \sum_m c_m(d_m)
\end{equation}
for some $k_i \in K_n^l$, $k_j\in K_n^{<l}$, $d_m \in D_n^*$, and $a_i,b_j,c_m \in \A(n)$.
Note that there can be no $d_m $ of dimension greater than $l$ by Lemma \ref{sq1}.

Consider the map induced by the one in Proposition \ref{comod}
\begin{equation*}
WH^* \to (\A/ \A \Sq^1) \otimes  M_n^{l}.
\end{equation*}
Then the first sum in \eqref{summer} is mapped to $\sum_i a_i \otimes k_i $, the second sum is mapped to 0, and the third sum is mapped to
\begin{equation*}
\sum_{ \dim d_m = l} c_m \otimes d_m .
\end{equation*}
Thus we get the equality $\sum_i a_i \otimes k_i = \sum_{ \dim d_m = l} c_m \otimes d_m$. But $K_n^*\cap D_n^* =0$, so this can only happen if $a_i , c_l \in \A \Sq^1$. But then $a_i\Sq^1 =0$ which implies that $a_i \in \A(n) \Sq^1$ since $\A(n)$ is free over $\A(0)$. This proves the claim.

It follows that we have a short exact sequence of $\A(n)$-modules
\begin{equation*}
\A(n)/\A(n) \Sq^1 \otimes K^l_n  \to WH^*/ \A(n)(D_n^*\oplus K_n^{<l}) \to WH^*/ \A(n)(D_n^*\oplus K_n^{< l+1}).
\end{equation*}
Introducing the notation
\begin{equation*}
T_{s,t}^l = \Tor_{s,t}^{\A(n)}(\Z/2^{\vee},WH^*/ \A(n)(D_n^*\oplus K_n^{<l})), 
\end{equation*}
this yields a long exact sequence of $\Tor$-groups for fixed $t$ and $l$
\begin{align}\label{torexact}
\dotsm \to \Tor_{s,t}^{\A(n)}(\Z/2^{\vee},\A(n)/\A(n)\Sq^1 &\otimes K^l_n )  \to T_{s,t}^l \to T_{s,t}^{l+1} \\
\to &
 \Tor_{s-1,t}^{\A(n)}(\Z/2^{\vee},\A(n)/\A(n)\Sq^1 \otimes K^l_n )\to \dotsm. \nonumber 
\end{align}
Note that 
\begin{equation*}
\Tor_{s,t}^{\A(n)}(\Z/2^{\vee}, \A(n)/\A(n)\Sq^1 \otimes  K^l_n) =\begin{cases} K_n^l &\textrm{ for } s=t-l,\\
0 &\textrm{ otherwise,}
\end{cases}
\end{equation*}
as one can see directly from a resolution, using the fact that $\A(n)$ is free over $\A(0)$. A direct construction of resolutions also shows that the composition
\begin{align}
\Tor_{s,t}^{\A(n)}(\Z/2^{\vee},\A(n)/\A(n)\Sq^1 \otimes K^{t-s}_n) \to  T_{s,t}^{t-s} 
 \to \Tor_{s,t}^{\A(n)}(\Z/2^{\vee},K_n^{t-s}) \cong  K_n^{t-s}\label{compiso}
\end{align}
is an isomorphism. In particular, the first map must be injective. Hence the long exact sequence \eqref{torexact} breaks up into short exact sequences. When $s=t-l$, this is the sequence
\begin{align}\label{exactK}
0 \to \Tor_{s,t}^{\A(n)}(\Z/2^{\vee},\A(n)/\A(n)\Sq^1 \otimes K^{t-s}_n) \to T_{s,t}^{t-s} \to T_{s,t}^{t-s+1} \to 0,
\end{align}
while for all other $s$, we get isomorphisms
\begin{equation}\label{toriso}
T_{s,t}^l \cong T_{s,t}^{l+1}.
\end{equation}

We now wish to compute $\Tor_{s,t}^{\A(n)}(\Z/2^{\vee},WH^*/ \A(n) D_n^*)$ for $s$ and $t$  fixed. 
For $l> t-s$,
$T_{s,t}^l =0$
because $WH^*/\A(n)(D_n^*\oplus K_n^{<l})$ is zero below degree $l$. Thus, \eqref{exactK} and \eqref{toriso}  yield isomorphisms for all $l\leq t-s$
\begin{equation}\label{t-l-s}
\Tor_{s,t}^{\A(n)}(\Z/2^{\vee},\A(n)/\A(n)\Sq^1 \otimes K^{t-s}_n) \cong T_{s,t}^{t-s} \cong T_{s,t}^{l}. 
\end{equation}

Furthermore,
\begin{equation}\label{isotor}
\Tor_{s,t}^{\A(n)}(\Z/2^{\vee},WH^*/ \A(n) D_n^*) \to T_{s,t}^l
\end{equation} 
is an isomorphism for $l$ sufficiently small. More precisely, the kernel of
\begin{equation*}
WH^*/  \A(n) D_n^* \to WH^* / \A(n)(D_n^*\oplus K_n^{<l})
\end{equation*}
is zero above dimension $l + \dim \A(n)$ and thus one can choose a resolution such that the $s$th term is zero above dimension $l + (s+1)\dim \A(n)$. So if $l$ is so small that $t>l + (s+1)\dim \A(n)$, the $\Tor$-groups of the kernel vanish, and the long exact sequence of $\Tor $-groups yields the isomorphism. 

In combination with \eqref{t-l-s}, \eqref{isotor} shows that the first map in
\begin{equation*}
\Tor_{s,t}^{\A(n)}(\Z/2^{\vee},WH^*/ \A(n) D_n^*) \to T_{s,t}^{t-s}
\to \Tor_{s,t}^{\A(n)}(\Z/2^{\vee}, K_n^{t-s})\cong K_n^{t-s}
\end{equation*}
is an isomorphism, and the second is an isomorphism by \eqref{compiso} and \eqref{t-l-s}.
\end{proof}

\begin{lem}\label{torK}
\begin{equation*}
\Tor_{s,t}^{\A(n)} (\Z/2^\vee,WH^*) \cong 
\begin{cases}
M_n^{t-s} &\textrm{ for } s=0,\\
K_n^{t-s} &\textrm{ for } s>0.
\end{cases}
\end{equation*}
For all $s>0$, the projection $WH^*\to K_n^{t-s}$ induces an isomorphism 
\begin{equation*}
\Tor_{s,t}^{\A(n)}(\Z/2^{\vee},WH^*) \to \Tor_{s,t}^{\A(n)}(\Z/2^{\vee}, K_n^{t-s}).
\end{equation*}
\end{lem}

\begin{proof}
By Lemma \ref{sq1}, the map 
\begin{equation*}
\A(n) \otimes D_n^*  \to WH^*
\end{equation*}
is injective, so there is a long exact sequence
\begin{align*}
\to \Tor_{s,t}^{\A(n)}(&\Z/2^{\vee},\A(n)\otimes D_n^*) \to \Tor_{s,t}^{\A(n)}(\Z/2^{\vee},WH^*)\\
 \to &\Tor_{s,t}^{\A(n)}(\Z/2^{\vee},WH^*/ (\A(n)\otimes D_n^* )) \to \Tor_{s-1,t}^{\A(n)}(\Z/2^{\vee},\A(n) \otimes D_n^*) \to.
\end{align*}
The $s=0$ part is the short exact sequence $0\to D_n^* \to M_n^* \to K_n^*\to 0$, and for $s>0$, we get isomorphisms
\begin{equation*}
 \Tor_{s,t}^{\A(n)}(\Z/2^{\vee},WH^*) \to \Tor_{s,t}^{\A(n)}(\Z/2^{\vee},WH^*/ (\A(n)\otimes D_n^*)).
\end{equation*}
The claim now follows from Lemma \ref{modD}.
\end{proof}

Define
\begin{align*}
M^*=&WH^*/\A^{>0} WH^*\\
K^k=&\Ker(\Sq^1: M^k \to WH^* / \A WH^{<k}).
\end{align*}
The projection $M^*_n \to M_{n+1}^*$ takes $K_n^*$ to $K_{n+1}^*$. Clearly, 
\begin{align*}
M^* &= \varinjlim_n M_n^*\\
K^* &= \varinjlim_n K_n^*. 
\end{align*}

We can now describe the $\Ext$ groups of $WH^*$. To determine extensions, recall the multiplicative structure of Adams spectral sequences.  
Let $h_0 \in \Ext^{1,1}_{\A}(\Z/2,\Z/2)$ be the generator.  Then there is a map $\Ext^{s,t}_{\A}(WH^*, \Z/2) \to \Ext^{s+1,t+1}_{\A}(WH^*, \Z/2)$ given by multiplication $h_0$. See e.g.\ \cite{mccleary} for the definition. 

\begin{thm}
\begin{align*}
\Tor^{s,t}_{\A}(\Z/2^\vee,WH^*)&\cong
\begin{cases}
M^t &\textrm{ for }s=0,\\
K^{t-s} &\textrm{ for } s>0.
\end{cases}\\
\Ext^{s,t}_{\A}(WH^*, \Z/2)&\cong 
\begin{cases}
(M^t)^\vee & \textrm{ for }s=0,\\
(K^{t-s})^\vee & \textrm{ for }s>0.
\end{cases}
\end{align*}
For $s>0$, 
\begin{equation*}
h_0^s: \Ext^{0,t-s}_{\A}(WH^*, \Z/2) \to \Ext^{s,t}_{\A}(WH^*, \Z/2)
\end{equation*}
is the projection $(M^{t-s})^\vee \to (K^{t-s})^\vee$.
\end{thm}

\begin{proof}
We first determine $\Tor_{s,t}^{\A}(\Z/2^\vee,WH^*)$. Obviously, $\Tor_{0,t}^{\A}(\Z/2^\vee,WH^*) \cong M^t$ by definition. For higher $s$, we must compute  $\varinjlim_n \Tor_{s,t}^{\A(n)}(\Z/2^\vee,WH^*)$, so we need to see what the maps in the direct system look like.

Suppose $M_n^* = D_n^* \oplus K_n^*$ is any splitting and that we have chosen $M_n^* \to WH^*$.
Look at the composite map
\begin{equation*}
WH^* \to M_n^{t-s} \to K_n^{t-s}.
\end{equation*}
This induces a diagram on $\Tor$-groups
\begin{equation*}
\xymatrix{
{\Tor_{s,t}^{\A(n)}(\Z/2^\vee,WH^*)}\ar[r]& {\Tor_{s,t}^{\A(n)}(\Z/2^\vee,M_n^{t-s})}\ar[r]\ar[d]^{\cong} &{\Tor_{s,t}^{\A(n)}(\Z/2^\vee,K_n^{t-s})}\ar[d]^{\cong}\\
{}&{M_n^{t-s}}\ar[r]&{K_n^{t-s}.}
}
\end{equation*}
The composition $\psi: \Tor_{s,t}^{\A(n)}(\Z/2^\vee,WH^*)\to K_n^{t-s}$ is an isomorphism by Corollary~\ref{torK}.

We claim that the image $K_n'$ of ${\Tor_{s,t}^{\A(n)}(\Z/2^\vee,WH^*)}$ in $M_n^{t-s}$ is exactly the subspace ${K_n^{t-s}}$.
Indeed, assume that $x\in K_n'$ is not contained in $K_n^{t-s}$. Hence we may choose $D_n^{t-s}$ such that $x\in D_n^{t-s}$. Then $\psi(x)=0$, which is a contradiction. Thus $K_n' \subseteq K_n^{t-s}$, and by surjectivity of $\psi$ they must be equal. 

The composition $WH^* \to M_n^{t-s} \to M_{n+1}^{t-s}$ yields a diagram of $\Tor$-groups
\begin{equation*}
\xymatrix{{\Tor_{s,t}^{\A(n)}(\Z/2^\vee,WH^*)}\ar[r]\ar[d]&{\Tor_{s,t}^{\A(n)}(\Z/2^\vee,M_{n}^{t-s})}\ar[d]\\
{\Tor_{s,t}^{\A(n+1)}(\Z/2^\vee,WH^*)}\ar[r]&{ \Tor_{s,t}^{\A(n+1)}(\Z/2^\vee,M_{n+1}^{t-s}).}
}
\end{equation*}
Here the right vertical map is the projection $M_{n}^{t-s}\to M_{n+1}^{t-s}$ and the left vertical map is the map of subspaces $K_n \to K_{n+1}$. 

Taking the direct limit over $n$ shows that the injection
\begin{equation*}
\Tor_{s,t}^{\A}(\Z/2^\vee,WH^*) \to \Tor_{s,t}^{\A }(\Z/2^\vee,M^{t-s})
\end{equation*}
has image exactly $K^{t-s}$. 

Dualizing, we see that
\begin{equation}\label{extM}
\Ext^{s,t}_{\A}(M^{t-s}, \Z/2) \to \Ext^{s,t}_{\A}(WH^*, \Z/2)
\end{equation}
is an isomorphism for $s=0$ and is exactly the map $(M^{t-s})^\vee \to (K^{t-s})^\vee$ otherwise.
But the map \eqref{extM} commutes with multiplication by $h_0$, and
\begin{equation*}
\Ext^{s,t}_{\A}(M^{t-s}, \Z/2) \cong h_0^s\Ext^{0,t-s}_{\A}(M^{t-s}, \Z/2)=h_0^s (M^{t-s})^\vee.
\end{equation*}
\end{proof}


The above yields a first description of the $E^2$-term of the spectral sequence. However, we would like a more explicit description of $M^*$ and $K^*$, as in the unoriented case.

 A direct computation shows that $\Hom_\A(H^*,\Z/2)$ contains the element $\xi_2^2 \in \Hom_\A^4(H^*(MTO),\Z/2)$ that takes the value 1 on both $w_4$ and $w_2^2$. The proof of Theorem \ref{xi2} carries over to show that:
\begin{prop}
The homomorphism $\xi_2^2$ is invertible in
$\Hom_\A(WH^*,\Z/2)$ and this becomes a module over $\Hom_\A(H^*,\Z/2)[\xi_2^{-2}]$.
\end{prop}

This certainly yields an infinite family of elements in $(M^*)^\vee$. We can say a bit more about the size and structure of $(K^*)^\vee$ and its complement 
\begin{equation*}
B^k = \Ker(\Hom_\A^k(WH^*,\Z/2) \to \Hom_\A^k(K^*,\Z/2)) = \Ker( (M^k)^\vee \to (K^k)^\vee).
\end{equation*}

\begin{prop} \label{Bhom}
$\xi_2^2 : \Hom_\A^k(WH^*,\Z/2) \to \Hom_\A^{k+4}(WH^*,\Z/2)$ restricts to an isomorphism $B^k \to B^{k+4}$ with inverse $\xi_2^{-2}$.
\end{prop}

\begin{proof}
Let $\xi \in B^k$. We claim that $\xi^{2n}_2 \cdot \xi \in B^{k+ 4n}$ for all $n \in \Z$. That is, $\xi^{2n}_2\cdot \xi(x)=0$ for all $x \in K^{k+ 4n}$. Note that it is enough to show this for positive $n$ because
 \begin{equation*}
\xi^{2n}_2 \cdot \xi(x) = \xi^{2n + 2^{N-1}}_2 \cdot \xi(t^{2^N}x)
 \end{equation*}
 and $x \in K^{k+ 4n}$ implies $t^{2^N}x \in K^{k+ 4n + 2^N}$ for $N$ large enough. Here multiplication by $t$ is as defined in Definition \ref{tmult}.

But for positive $n$, $\xi^{2n}_2 \cdot \xi$ factors as $\A$-linear maps 
\begin{equation*}
WH^* \to H^* \hat{\otimes} WH^* \xrightarrow{\xi_2^{2n} \otimes \textrm{id}} (\Z/2)^{4n} {\otimes} WH^* \xrightarrow{1\otimes\xi} \Z/2.
\end{equation*}

Assume that $x$ maps to $1\otimes x'$ in $(\Z/2)^{4n} {\otimes} WH^*$, and hence $\Sq^1(x)$ maps to $  1 \otimes \Sq^1(x')$.
But $\Sq^1(x)$ decomposes as $\sum_i a_i (x_i)$, so if $x_i $ maps to $  1 \otimes x_i'$, then
\begin{equation*}
1 \otimes \Sq^1(x') = \sum_i 1 \otimes a_i (x_i').
\end{equation*}
Thus $\xi(x') =0 $, proving the claim. The proposition now follows because $\xi_2^2$ restricts to $B^k \to B^{k+4}$ and $\xi_2^{-2}$ restricts to an inverse.
\end{proof}

Define
\begin{align*}
K(H)^k &= \Ker(\Sq^1 : H^k/\A H^{<k} \to H^{k+1}/\A H^{<k})\\
B(H)^k &= \Ker(\Hom_\A^k(H^*,\Z/2) \to \Hom_\A^k(K(H)^*,\Z/2)).
\end{align*}

\begin{lem}\label{K(H)}
The map $WH^*/\A^{>0} WH^* \to H^*/\A^{>0} H^*$ maps $K^*$ surjectively onto $K(H)^*$.
\end{lem}

\begin{proof}
As representatives for a basis of $K(H)^*$ we may take all products $\prod {w}_{2k}^{2n_k}$, c.f.~\cite{switzer}, Chapter 20. Such a $w^2 = \prod {w}_{2k}^{2n_k}$ lifts to $w^2 \tilde{w}_0 + \Sq^2(w^2)\tilde{w}_{-2}$, which represents an element of $M^*$. In fact, this lies in $K^*$ because
\begin{equation*}
\Sq^1(w^2\tilde{w}_0 + \Sq^2(w^2)\tilde{w}_{-2}) = w^2\tilde{w}_1 + \Sq^2(w^2)\tilde{w}_{-1}= \Sq^2(w^2\tilde{w}_{-1}).
\end{equation*}
Thus $K^* \to K(H)^*$ is surjective. 
\end{proof}

\begin{prop}
\begin{equation*}
\Hom_\A(H^*,\Z/2)[\xi_2^{-2}] \cap B^* = B(H)^*[\xi_2^{-2}].
\end{equation*}
\end{prop}

\begin{proof}
The inclusion $\supseteq $ follows from Proposition \ref{Bhom}.

Assume $ \xi_2^{2n} \cdot \xi \in B^*$ for some $\xi \in \Hom_\A(H^*,\Z/2)$. Then also $ \xi_2^{2n+2N} \cdot \xi \in B^*$, and for $N$ large, 
\begin{equation*}
 \xi_2^{2n+2N} \cdot \xi \in \Hom_\A(H^*,\Z/2) \cap B^*.
\end{equation*}
 But then $ \xi_2^{2n+2N} \cdot \xi$ vanishes on $K^*$. Thus by Lemma \ref{K(H)}, $ \xi_2^{2n+2N} \cdot \xi \in B(H)^*$.
\end{proof}

\begin{cor}
$M^*/K^*$ is infinite in all dimensions and $K^*$ is infinite in all dimensions divisible by 4.
\end{cor}

We do not know whether $K^*$ is trivial in dimensions not divisible by $4$ as is the case for $K(H)^*$.

\begin{cor}
The map $MT(d) \to \widehat{MT}(d)_2^\wedge$ induces an injection on the $E_2^{s,t}$-term of the Adams spectral sequences for $t-s \leq d$.  
\end{cor}

\begin{proof}
This follows on $E^{0,t}_2=\Ext_\A^{0,t}=\Hom_\A^{0,t}$ because $WH^* \to H^*$ is surjective. For $s>0$, the map $\Ext_\A^{s,t}(H^*,\Z/2) \to \Ext_\A^{s,t}(WH^*,\Z/2)$ is $(K(H)^{t-s})^\vee\to (K^{t-s})^\vee$, which is injective by Lemma \ref{K(H)}.
\end{proof}

Finally, we have not been able to show that the spectral sequence collapses. It seems natural since it does so in both the unoriented case and for $H^*$. Again a better understanding of $M^*$ and $K^*$ would be helpful. It would be enough to show that $\Hom_\A(H^*,\Z/2)[\xi_2^{-2}]$ detects all of $M^*$ as in the unoriented case:
\begin{thm}
Assume that $\Hom_\A(H^*,\Z/2)[\xi_2^{-2}]$ detects all elements of $M^*$. Then the space $K^*$ is trivial except in dimensions divisible by 4. In particular, the Adams spectral sequence collapses.
\end{thm}
\begin{proof}
Let $[x] \in K^*$. Then there is a $\xi \cdot \xi_2^{-2n} \in \Hom_\A(H^*,\Z/2)[\xi_2^{-2}]$ such that $\xi \cdot \xi_2^{-2n}(x) = 1$. Let $\Sq^1(x)= \sum_i a_i(x_i)$ for suitable $a_i \in \A$. Then for a sufficiently large $N$
\begin{align*}
\xi \cdot \xi_2^{2^{N-1}-2n}(t^{2^N} x) &= 1\\
\Sq^1(t^{2^N}x)&= \sum_i a_i(t^{2^N}x_i)
\end{align*}
where $\xi \cdot \xi_2^{2^{N-1}-2n} \in \Hom_\A(H^*,\Z/2)$. This means that the image of $t^{2^N} x$ in $H^*$ is indecomposable and $\Sq^1(t^{2^N}x)$ is decomposable over $\A^{>1}$. Hence its dimension must be divisible by 4. Thus also $x$ must have dimension divisible by 4.

This implies that there can only be higher $\Ext$-groups in dimensions divisible by 4. Then it follows from the multiplicative structure that there can be no non-trivial differentials.
\end{proof}

%
%

\end{document}